\documentclass[11pt]{amsart}
\usepackage{amsmath, amssymb}
\usepackage{amsfonts, tikz}
\usepackage[pdftex]{hyperref}
\usepackage{mathrsfs}
\usepackage[arrow,matrix,curve,cmtip,ps]{xy}

\usepackage{amsthm}

\allowdisplaybreaks

\newtheorem{theorem}{Theorem}[section]
\newtheorem{lemma}[theorem]{Lemma}
\newtheorem{proposition}[theorem]{Proposition}
\newtheorem{corollary}[theorem]{Corollary}

\newtheorem*{theorem*}{Theorem}
\theoremstyle{remark}
\newtheorem{remark}[theorem]{Remark}
\newtheorem{definition}[theorem]{Definition}
\newtheorem{example}[theorem]{Example}

\numberwithin{equation}{section}

\newcommand{\Z}{\mathbb{Z}}
\newcommand{\N}{\mathbb{N}}

\newcommand{\C}{\mathbb{C}}

\newcommand{\A}{\mathcal{A}}
\newcommand{\G}{\mathcal{G}}
\newcommand{\Pa}{\partial}

\newcommand{\F}{\mathcal{F}}
\newcommand{\Sfin}{\Sigma_\A^\textnormal{fin}}
\newcommand{\Sinf}{\Sigma_\A^\textnormal{inf}}
\newcommand{\Xfin}{X^\textnormal{fin}}
\newcommand{\Xinf}{X^\textnormal{inf}}
\newcommand{\0}{\vec{0}}

\begin{document}
\title{One-Sided Shift Spaces Over Infinite Alphabets}

\author{William Ott 
}

\author{Mark Tomforde 
}

\author{Paulette N. Willis 
}

\address{Department of Mathematics \\ University of Houston \\ Houston, TX 77204-3008 \\USA}
\email{ott@math.uh.edu}

\address{Department of Mathematics \\ University of Houston \\ Houston, TX 77204-3008 \\USA}
\email{tomforde@math.uh.edu}

\address{Department of Mathematics \\ University of Houston \\ Houston, TX 77204-3008 \\USA}
\email{pnwillis@math.uh.edu}

\thanks{This work was partially supported by a grant from the Simons Foundation (\#210035 to Mark Tomforde) and also partially supported by NSF Mathematical Sciences Postdoctoral Fellowship DMS-1004675.}

\date{\today}

\subjclass[2010]{37B10, 46L55}

\keywords{Symbolic dynamics, one-sided shift spaces, sliding block codes, infinite alphabets, shifts of finite type, $C^*$-algebras}

\begin{abstract}

We define a notion of (one-sided) shift spaces over infinite alphabets.  Unlike many previous approaches to shift spaces over countable alphabets, our shift spaces are compact Hausdorff spaces.  We examine shift morphisms between these shift spaces, and identify three distinct classes that generalize the shifts of finite type.  We show that when our shift spaces satisfy a property that we call ``row-finite", shift morphisms on them may be identified with sliding block codes.  As applications, we show that if two (possibly infinite) directed graphs have edge shifts that are conjugate, then the groupoids of the graphs are isomorphic, and the $C^*$-algebras of the graphs are isomorphic.
\end{abstract}

\maketitle

\tableofcontents

\section{Introduction}

In symbolic dynamics one begins with a set of symbols and considers spaces consisting of sequences of these symbols that are closed under the shift map.  There are two approaches that are used: one-sided shift spaces that use sequences of symbols indexed by $\N$, and two-sided shift spaces that use bi-infinite sequences indexed by $\Z$.  In this paper, we shall be concerned exclusively with one-sided shifts.

In the classical construction of a one-sided shift space, one begins with a finite set $\A$ (called the \emph{alphabet} or \emph{symbol space}) and then considers the set $$\A^\N := \A \times \A \times \ldots$$ consisting of all sequences of elements of $\A$.  If we give $\A$ the discrete topology, then $\A$ is compact (since $\A$ is finite), and Tychonoff's theorem implies that $\A^\N$ with the product topology is also compact.   In addition, the shift map $\sigma : \A^\N \to \A^\N$ defined by $\sigma(x_1x_2x_3 \ldots) := x_2x_3x_4\ldots$ is continuous. The pair $(\A^\N, \sigma)$ is called the (one-sided) \emph{full shift space}, and a \emph{shift space} is defined to be a pair $(X,\sigma|_X)$ where $X$ is subset of $\A^\N$ such that $X$ is closed and $\sigma(X) \subseteq X$.  Since $X$ is a closed subset of a compact space, $X$ is also compact.  In the analysis of shift spaces the compactness plays an essential role, and many fundamental results rely on this property.

Attempts to develop a theory of shift spaces when the alphabet $\A$ is infinite (even countably infinite) have often been stymied by the fact that the spaces considered are no longer compact --- and worse yet, not even locally compact.  For instance, if one takes a countably infinite set $\A = \{a_1, a_2, \ldots \}$, one can give $\A$ the discrete topology and consider the space
$$\A^\N := \A \times \A \times \ldots$$
with the product topology.  In this situation, the shift map $\sigma : \A^\N \to \A^\N$ defined by $\sigma(x_1x_2x_3 \ldots) := x_2x_3x_4\ldots$ is continuous.  However, the space $\A^\N$ is no longer compact or even locally compact.  For example, any open set $U$ in $\A^\N$ must contain a basis element of the form $$Z(x_1 \ldots x_m) = \{ x_1 \ldots x_m z_{m+1} z_{m+2} \ldots \in \A^\N : \text{ $z_k \in \A$ for $k \geq m+1$} \},$$  and if we define $x^n := x_1 \ldots x_m a_n a_n a_n \ldots$, then $\{ x^n \}_{n=1}^\infty$ is a sequence in $Z(x_1 \ldots x_m)$ without a convergent subsequence.  Hence the closure of $\emph{U}$ is not (sequentially) compact, and $\A^\N$ is not locally compact.  Therefore, if we define a shift space over $\A$ to be a pair $(X,\sigma|_X)$ where $X$ is a closed subset of $\A^\N$ with the property that $\sigma(X) \subseteq X$, then the set $X$ will be a closed, but not necessarily compact, subset of $\A^\N$.  This lack of compactness makes it difficult to establish results for such subspaces, and as a result this approach to shift spaces over countable alphabets has encountered difficulties.

The purpose of this paper is to give a new definition for the (one-sided) full shift and its subshifts when the alphabet $\A$ is infinite.  In this new definition the full shift and all shift spaces are compact, and this will allow techniques from the classical theory of shifts over finite alphabets to be more readily generalized to this setting.  It is our hope that this new definition will allow for applications to dynamics that are unavailable using current methods.  Furthermore, our new definition reduces to the classical definition when $\A$ is finite.

The key idea of our new definition of the full shift is to begin with an infinite alphabet $\A$ that we endow with the discrete topology.  We then let $\A_\infty = \A \cup \{ \infty \}$ denote the one-point compactification of $\A$.  Since $\A_\infty$ is compact, the product space 
$$X_\A := \A_\infty \times \A_\infty \times \ldots$$
is compact.  However, we do not want to take $X_\A$ as our definition of the full shift, since it includes sequences that contain the symbol $\infty$, which is not in our original alphabet.  Therefore, we shall consider an identification of elements of $X_\A$ with infinite and finite sequences of elements in $\A$.  Specifically, we do the following:  If $x = x_1 x_2 \ldots \in X_\A$ has the property that $x_i \neq \infty$ for all $i \in \N$, then we do nothing and simply consider this as an infinite sequence of elements of $\A$.  If $x = x_1 x_2 \ldots \in X_\A$ has an $\infty$ occurring, we consider the first place that such an $\infty$ appears; for example, write $x = x_1 \ldots x_n \infty \ldots$ with $x_i \neq \infty$ for $1 \leq i \leq n$ and identify $x$ with the finite sequence $x_1 \ldots x_n$.  In this way we define an equivalence relation $\sim$ on $X_\A$ such that the quotient space $X_\A / \sim$ of all equivalence classes is identified with the collection of all sequences of symbols from $\A$ that are either infinite or finite (details of this equivalence relation are described in Section~\ref{top-def-subsec}).  We let $\Sigma_\A$ denote the set of all finite and infinite sequences of elements of $\A$, and using the identification of $\Sigma_\A$ with $X_\A / \sim$, we give $\Sigma_\A$ the quotient topology it inherits from $X_\A$.  While quotient topologies are in general not well behaved, we can prove that with this topology the space $\Sigma_\A$ is both compact and Hausdorff.  Moreover, the shift map $\sigma : \Sigma_\A \to \Sigma_\A$, which simply removes the first entry from any sequence, is a map on $\Sigma_\A$ that is continuous at all points except the empty sequence.  We then define the one-sided full shift to be the pair $(\Sigma_\A, \sigma)$.

Next we define shift spaces.  As usual, we want to consider subsets of $\Sigma_\A$ that are closed and invariant under $\sigma$; however, we also want an additional property. Motivated by classical edge shifts of finite graphs having no sinks, we require that any finite sequence in the subset can be extended to an infinite sequence in the subset with infinitely many choices of the next symbol (or, in more precise language: for any 
finite sequence $w$ in our shift space there exist sequences of the form $w a x$ in the shift space for infinitely many distinct $a \in A$).  We call this the ``infinite-extension property", and a precise definition is given in Definition~\ref{inf-ext-def}.  We thus define a \emph{shift space} to be a pair $(X,\sigma|_X)$ where $X$ is a subset of $\Sigma_\A$ such that $X$ is closed, $\sigma(X) \subseteq X$, and $X$ has the ``infinite-extension property".  As closed subsets of $\Sigma_\A$, our shift spaces will necessarily be compact.  In this paper we lay the groundwork for the study of these spaces, and a study of morphisms between them.  We hope that this approach will be useful for extending certain aspects of symbolic dynamics to the case of infinite alphabets, as well as allowing methods from symbolic dynamics to be applied to graph $C^*$-algebras of graphs with infinitely many edges.

Many of our inspirations for the topology on the set $\Sigma_\A$ come from the theory of graph $C^*$-algebras and the study of the boundary path space of a graph.  Since the fundamental work of Cuntz and Krieger in \cite{Cun, CK} it has been known that the Cuntz-Krieger algebras (i.e., $C^*$-algebras associated with finite graphs having no sinks or sources) are intimately related to the shift spaces of the graphs --- and, in particular, that conjugacy of the one-sided shift spaces of two graphs implies isomorphism of the $C^*$-algebras of those graphs.

This relationship has been explored in many contexts throughout the three decades since Cuntz and Krieger's work.  The ideas that are most influential for us in defining a notion of one-sided shift spaces for infinite alphabets are Paterson's work on (topological) groupoids for infinite graphs \cite{Pat}, Paterson and Welch's construction of a product of locally compact spaces that satisfies a Tychonoff theorem~\cite{PW}, Yeend's work on groupoids of topological graphs \cite{Yeend-thesis, Yeend}, and Webster's work on path spaces and boundary path spaces of graphs \cite{Webster-thesis, Web}.  These constructions, which are all related, provide motivation for our construction of the space $\Sigma_\A$ --- both as a set and as a topological space.

In the past few decades there have been numerous efforts by various authors to define and study analogues of shift spaces over countable alphabets, most commonly in the context of countable-state Markov chains (or equivalently, shifts coming from countable directed graphs or matrices).   For the reader's benefit we mention a few of these:  The paper \cite{GS} by Gurevich and Savchenko contains a detailed survey of the theory of symbolic Markov chains over an infinite alphabet as well as several expositions of results of the authors; Petersen has shown in \cite{Pet} that there is no Curtis-Hedlund-Lyndon theorem for factor maps between tiling dynamical systems; in the paper \cite{FF} D.~Fiebig and U.~Fiebig examine continuous shift-commuting maps from transitive countable-state Markov shifts into compact subshifts; and Wagoner in \cite{Wag87, Wag88} has studied the group of uniformly continuous shift-commuting maps with uniformly continuous inverse on two-sided Markov shifts over countable alphabets.  Significant progress has also been made on the development of thermodynamic formalism for symbolic Markov chains with countably many states (e.g.~\cite{BBG07, CS, FFY, GS, IY, MU, Sar99}).  Phase transitions have been investigated in this context (e.g.~\cite{PZ, Sar01, Sar06}), and some countable-state Markov shifts have been classified up to almost isomorphism.  Boyle, Buzzi, and G\'{o}mez~\cite{BBG06} show that two strongly positive recurrent Markov shifts are almost isomorphic if and only if they have the same entropy and period.  Markov towers, abstract models resembling countable-state Markov chains, encode statistical properties of many dynamical systems that possess some hyperbolicity.  Young~\cite{You98, You99} introduces the abstract tower model and uses it to prove that correlations in the finite-horizon Lorentz gas decay at an exponential rate.  We also mention the work of Exel and Laca in \cite{EL}, where they construct ``Cuntz-Krieger algebras for infinite matrices".  Their realization of these $C^*$-algebras as a crossed product allows them to identity the spectrum of the diagonal algebra with a compactification of the set of infinite paths, and in the last sentence of the introduction of \cite{EL} the authors suggest this space may be a suitable replacement for the infinite path space in the study of topological Markov chains with infinitely many states.

The papers listed in the previous paragraph show that there have been many different approaches to shift spaces over infinite alphabets, and even many different definitions of what a shift space (or Markov chain) over an infinite alphabet should be.  The results produced by these different theories suggest the possibility that there is no one correct definition of a shift space over an infinite alphabet, but rather different definitions that are useful for different purposes.  (A remark to this effect is explicit in \cite{GS}, where the authors emphasized this viewpoint with a descriptor for countable-state shifts of symbolic Markov chains rather than topological Markov chains, and this idea is also alluded to in \cite{BBG06}.)  

Our definition of a shift space over an infinite alphabet provides a new addition to the panoply of definitions that have come before and a new avenue for exploration.  The novel features of our definition are (1) our shift spaces are compact, which allows for many topological results from the finite alphabet case to be generalized to our spaces, and (2) our shift spaces are intimately related to path spaces of directed graphs, and as a result have applications to Cuntz-Krieger algebras and $C^*$-algebras of graphs.  This second feature, in particular, shows that among the myriad definitions given by prior authors, our definition of a shift space seems to be the most advantageous for working with $C^*$-algebras.

This paper is organized as follows:  In Section~\ref{fullshift-sec} we give a formal definition of our one-sided shift space $(\Sigma_\A, \sigma)$ for an infinite alphabet $\A$.  Specifically, in Section~\ref{top-def-subsec} we define $\Sigma_\A$ as a topological space, and prove that it is compact and Hausdorff.  In Section~\ref{basis-top-subsec} we show that $\Sigma_\A$ has a basis of ``generalized cylinder sets", and we use this basis to get a better understanding of the topology and describe pointwise convergence in $\Sigma_\A$.   In Section~\ref{metrics-sec} we show that when $\A$ is countable (and hence $\Sigma_\A$ is second countable), there exists a natural family of metrics on $\Sigma_\A$ that produces our topology.  In Section~\ref{shift-map-subsec} we prove that the shift map $\sigma: \Sigma_\A \to \Sigma_\A$ is continuous at all points except the empty sequence $\0 \in \Sigma_\A$, and the restriction of $\sigma_\A$ to $\Sigma_\A \setminus \{ \0 \}$ is a local homeomorphism.

In Section~\ref{shift-spaces-sec} we define shift spaces as closed subspaces of $\Sigma_\A$ that are invariant under the shift map $\sigma$ and have the ``infinite-extension property" (see Definition~\ref{inf-ext-def}).  The infinite-extension property, which is vacuously satisfied in the finite alphabet case, ensures that finite sequences in shift spaces can be extended to infinite sequences and this extension can be done with an infinite number of choices for the next symbol.  (In dynamics terms, this is often described as saying that finite sequences end at symbols with ``infinite followers sets"; in graph terms it is often said the sequences end at vertices that are ``infinite emitters".)  We show that with our definition, shift spaces can be described in terms of \emph{forbidden blocks}, and any shift space may be recovered from those blocks that do not appear in any of its (finite or infinite) sequences.  We also establish some basic properties of shift spaces, and identify two important classes: the \emph{finite-symbol} shift spaces, which can be realized as shift spaces over finite alphabets as in the classical case, and the \emph{row-finite} shift spaces in which every symbol has a finite number of symbols that may follow it.  In particular, row-finite spaces have no nonempty finite sequences, and thus every element in a row-finite shift space is either an infinite sequence or the empty sequence $\0$.  We conclude Section~\ref{shift-spaces-sec} with several characterizations of the finite-symbol shift spaces and the row-finite shift spaces.

In Section~\ref{shift-morphisms-sec} we define shift morphisms as maps between shift spaces that are continuous, commute with the shift, and preserve lengths of sequences. These shift morphisms appear to be more complicated than the ``sliding block codes" that arise in the finite alphabet setting. We establish some basic results in this section, and conclude the section by defining \emph{conjugacy}, which is the notion of isomorphism in our category.

In Section~\ref{analogues-of-SFT-sec} we consider analogues of shifts of finite type.  In the finite alphabet case, it is well known that a shift space is a shift of finite type (i.e., described by a finite set of forbidden blocks) if and only if it is conjugate to the edge shift coming from a finite graph with no sinks if and only if it is an $M$-step shift (i.e., the shift is described by a set of forbidden blocks all of length $M+1$).  We show that in the infinite alphabet case these three classes are distinct --- namely, the conjugacy classes of shifts of finite type, edge shifts, and $M$-step shifts are distinct.  We describe how these classes are related and identify the class of edge shifts as the class that is the most reasonable for extending classical results for shifts of finite type to the infinite alphabet situation.

In Section~\ref{row-finite-analogues-of-SFT-sec} we analyze our three generalizations of shifts of finite type in the row-finite setting.  Here things are a bit nicer: We show that the only row-finite shifts of finite type are the finite-symbol shifts, and are thus covered by the classical case.  We also show that the class of row-finite edge shifts coincides with the class of row-finite $M$-step shifts.  Again, it is this class of row-finite edge shifts (equivalently, row-finite $M$-step shifts) that seems most  reasonable for extending classical results for shifts of finite type to the infinite alphabet situation.

In Section~\ref{SBC-row-finite-sec} we consider shift morphisms on row-finite shift spaces.  We show that in the row-finite setting all shift morphisms come from ``sliding block codes" (see Theorem~\ref{SBC-ff-shift-morphism-thm}).  Unlike the finite alphabet case, however, we classify these into two types: unbounded and bounded (see Definition~\ref{SBC-def}). The bounded sliding block codes are just like the sliding block codes in the finite alphabet case, but the unbounded sliding block codes require a sequence of $N$-block maps, one for each symbol, that are unbounded in $N$.  We show that, as in the finite alphabet case, bounded sliding block codes may be recoded to $1$-block codes (see Proposition~\ref{SBC-ff-shift-morphism-thm}).  We conclude Section~\ref{SBC-row-finite-sec} with a characterization of bounded sliding block codes in Proposition~\ref{bounded-fff-uniformly-cts-prop}.

In Section~\ref{C*-algebras-sec} we connect our ideas with $C^*$-algebras and give applications of our results.  In Section~\ref{C*-alg-groupoid-subsec} we show that if we have two (possibly infinite) graphs with no sinks, then conjugacy of the edge shifts of these graphs implies isomorphism of the $C^*$-algebras of the graphs (see Corollary~\ref{graph-C*-iso-cor}).  Indeed we are able to prove something slightly stronger:  conjugacy of the edge shifts of the graphs implies isomorphism of the graph groupoids (see Theorem~\ref{shifts-conjugate-implies-groupoids-iso-thm}).  We consider this strong supporting evidence that our definition of the one-sided edge shift given in this paper is the correct one in the context of working with $C^*$-algebras.  Given the long-standing relationship between symbolic dynamics and $C^*$-algebras, it is reassuring to see that these important implications still hold in the infinite alphabet case.  In Section~\ref{LPA-subsec} we establish as a corollary that if $E$ and $F$ are (possibly infinite) graphs with no sinks, then conjugacy of the edge shifts of these graphs implies isomorphism of the complex Leavitt path algebras $L_\C(E)$ and $L_\C(F)$.  In Section~\ref{row-finite-graph-alg-subsec} we show that when we have a bounded sliding block code between row-finite edge shifts, we can recode to a $1$-block map and obtain an explicit isomorphism between the graph $C^*$-algebras and also between the Leavitt path algebras over any field (see Theorem~\ref{ConjIso}).  Since all shift morphisms are bounded sliding block codes in the finite alphabet case, this implies that if $K$ is any field, and if $E$ and $F$ are finite graphs with no sinks and conjugate edge shifts, then the Leavitt path algebras $L_K(E)$ and $L_K(F)$ are isomorphic.

$ $

\noindent \textbf{Notation and Terminology:}  Throughout we take the natural numbers to be the set $\N = \{1, 2, 3 \ldots \}$.  The term \emph{countable} will mean either finite or countably infinite.    Since we are writing for two audiences that may have different backgrounds (symbolic dynamicists and $C^*$-algebraists) we strive to make the exposition as clear as possible, explain our motivations, and provide examples.  We do our best to be clear without being pedantic.  Throughout we will often choose terminology motivated by graph algebras (e.g., ``row-finite", ``sinks", ``infinite emitters") --- even though we know these are not the terms most dynamicists would choose.  We do this because the study of graphs, and the theory of their $C^*$-algebras, are where our motivation comes from, and we believe that interactions with graph $C^*$-algebras will be at the forefront of the applications of these ideas.

\section{The Full Shift over an Infinite Alphabet} \label{fullshift-sec}

In this section we define the one-sided full shift $(\Sigma_\A, \sigma)$ over a (possibly infinite) alphabet $\A$.  We shall first define the set $\Sigma_\A$, and topologize $\Sigma_\A$ in such a way that it is a compact Hausdorff space.  Afterward, we describe a  convenient basis for $\Sigma_\A$ that gives us a better understanding of the topology, and we use this basis to characterize sequential convergence in $\Sigma_\A$.  At the end of this section we define the shift map $\sigma : \Sigma_{\A} \to \Sigma_\A$ and show that it is continuous at all points except the empty sequence $\0 \in \Sigma_\A$.

\subsection{Definition of the Topological Space $\Sigma_\A$} \label{top-def-subsec}

Suppose that $\A$ is an infinite set, which we shall call an \emph{alphabet}.  The elements of $\A$ will be called \emph{letters} or \emph{symbols}.  

We define $\A^0 := \{ \0 \}$ where $\0$ is the \emph{empty sequence} consisting of no terms, and for each $k \in \N$ we define $$\A^k := \underbrace{\A \times \ldots \times \A}_{\text{$k$ copies}}$$ to be the product of $k$ copies of $\A$.  We also define $$ \A^\N := \A \times \A \times \ldots$$ to be the product of a countably infinite number of copies of $\A$.  Observe that the sets $\A^\N$ and $\A^k$ for $k \in \mathbb{N} \cup \{ 0 \}$ are pairwise disjoint.  

\begin{definition} \label{Sigma-X-def}
We define $\Sigma_{\A}$ to be the disjoint union 
\begin{equation*}
\Sigma_{\A} := \A^\N \cup \bigcup_{k=0}^\infty \A^k.
\end{equation*}

We refer to the elements of $\Sigma_{\A}$ as \emph{sequences}. (We use this terminology despite the fact that some of our sequences have a finite number of terms.)  We define a function $l : \Sigma_{\A} \to \{ 0, 1, 2, \ldots, \infty \}$ by $l(x) = \infty$ if $x \in \A^\N$ and $l(x) = k$ if $x \in \A^k$.  Note that the empty word $\0$ is the unique sequence with $l(\0) = 0$.  We call the value $l(x)$ the \emph{length} of the sequence $x$.  We define $\Sinf := \A^\N$ and call the elements of this set \emph{infinite sequences}, and we define $\Sfin := \bigcup_{k=0}^\infty \A^k$ and call the elements of this set \emph{finite sequences}.   If $x \in \Sigma_\A$, then when $0 < l(x) < \infty$ we denote the entries of $x$ as  $x = x_1 \ldots x_{l(x)}$ with $x_i \in \A$, and when $l(x) = \infty$, we denote the entries of $x$ as $x = x_1 x_2 x_3 \ldots$ with $x_i \in \A$.

At this point we wish to topologize $\Sigma_\A$.

\begin{definition} \label{A-infty-def}
If $\A$ is an infinite set, give $\A$ the discrete topology and define $\A_\infty := \A \cup \{ \infty \}$ to be minimal compactification of $\A$.  Since $\A$ is infinite, $\A_\infty := \A \cup \{ \infty \}$ is the one-point compactification of $\A$.  Note, in particular, that the topology on $\A_\infty$ is given by the collection:
$$\{ U : U \subseteq \A \} \cup \{ \A_\infty \setminus F : \text{$F$ is a finite subset of $\A$} \}.$$
so that the open sets of $\A_\infty$ are the subsets of $\A$ together with the complements of finite subsets of $\A$.
\end{definition}

Let $\A$ be an infinite set.  Then $\A_\infty$ is a compact Hausdorff space, and by Tychonoff's theorem the countably infinite product space $$X_\A := \A_\infty \times \A_\infty \times \ldots$$ is a compact Hausdorff space.  Define a function $$Q : X_\A \to \Sigma_{\A}$$ by 
\begin{equation*}
Q (x_1 x_2 \ldots ) = \begin{cases} \0 & \text{ if $x_1 = \infty$} \\ x_1 \ldots x_n & \text{ if $x_{n+1} = \infty$ and $x_i \neq \infty$ for $1 \leq i \leq n$} \\ x_1 x_2 x_3 \ldots & \text{ if $x_i \neq \infty$ for all $i \in \N$.} \end{cases}
\end{equation*}
Observe that $Q$ is surjective.  We give $\Sigma_A$ the quotient topology induced from $X_\A$ via the map $Q$; in particular, with this definition $U \subseteq \Sigma_A$ is open if and only if $Q^{-1}(U) \subseteq X_\A$ is open.
\end{definition}

\begin{remark}
The map $Q$ defines an equivalence relation on the space $X_\A$ by $x \sim y$ if and only if $Q(x) = Q(y)$.  Under this equivalence relation, two elements of $X_\A$ are equivalent precisely when they have all entries equal up to  the first appearance of the symbol $\infty$, and the equivalence class of such an $x \in X_\A$ is identified with the (finite or infinite) sequence in $\Sigma_\A$ having the same entries as $x$ up to the first appearance of $\infty$.  Note that it is quite possible that there are no occurrences of $\infty$ in $x$.
\end{remark}

\begin{example}
Suppose $\A = \{ a_1, a_2, \ldots \}$.  The elements 
\begin{align*}
x &:= a_1 a_2 \infty a_1 \infty \infty a_1 \ldots \in X_\A \\
y &:= a_1 a_2 \infty a_2 a_7 \infty a_2 \ldots  \in X_\A 
\end{align*}
are equivalent in $X_\A$ and each is identified with the finite sequence $a_1 a_2 \in \Sigma_{\A}$ of length $2$.  The element 
$$z:=a_1 a_2 a_3 a_4 \infty a_2 \ldots \in X_\A$$
is not equivalent to either $x$ or $y$ and is identified with the finite sequence $a_1 a_2 a_3 a_4 \in \Sigma_\A$ of length $4$.  Any element $x = x_1 x_2 x_3 \ldots \in X_\A$ with $x_i \neq \infty$ for all $i \in \N$ is identified with the infinite sequence $x_1 x_2 x_3 \ldots \in \Sigma_\A$, and no element of $X_\A$ other than $x$ itself is equivalent to $x$.  The element $$\infty \infty \infty \ldots \in X_\A$$ is identified with the empty sequence $\0 \in \Sigma_\A$.  In fact, every element $x \in X_\A$ such that $x_1= \infty$ is equivalent to $\infty \infty \infty \ldots$.
\end{example}

\begin{proposition}
The space $\Sigma_\A$ is a compact Hausdorff space.
\end{proposition}

\begin{proof}
Since the quotient map $Q : X_\A \to \Sigma_\A$ is continuous, and $X_\A$ is compact, it follows that $\Sigma_\A$ is compact.  To prove that $\Sigma_\A$ is Hausdorff, \cite[Proposition~5.4 of Appendix~A]{Massey} shows it suffices to show that the set $G := \{ (x,y) \in X_\A \times X_\A : Q(x) = Q(y) \}$ is closed in $X_\A \times X_\A$.  Suppose that $\{ (x^n, y^n) \}_{n=1}^\infty$ is a sequence of points in $G$ with $\lim_{n \to \infty} (x^n, y^n) = (x,y) \in X_{\A} \times X_{\A}$.  Write $x^n := x^n_1x^n_2 x^n_3 \ldots$ and $y^n := y^n_1 y^n_2 y^n_3 \ldots$ for each $n \in \N$.  Also write $x=x_1x_2x_3 \ldots$ and $y = y_1 y_2 y_3 \ldots$.  Then $\lim_{n\to \infty} x^n_i = x_i$ and  $\lim_{n\to \infty} y^n_i = y_i$ for all $i \in \N$.

For each $z = z_1 z_2 z_3 \ldots \in X_\A$, define $$L(z) :=  \begin{cases} 0 & \text{ if $z_1 = \infty$} \\ N & \text{ if $x_{N+1} = \infty$ and $x_i \neq \infty$ for $1 \leq i \leq N$} \\ \infty  & \text{ if $x_i \neq \infty$ for all $i \in \N$.} \end{cases}
$$
Note that $L(z) = l(Q(z))$.  Since $Q(x^n) = Q(y^n)$ for all $n \in \N$, we see that $L(x^n) = L(y^n)$ for all $n \in \N$.  We shall look at the sequence $\{ L(x^n) \}_{n=1}^\infty$, and consider two cases.

The first case is that $\{ L(x^n) \}_{n=1}^\infty$ is bounded.  Then, by passing to a subsequence, we may suppose that $L(x^n)$ is equal to a constant value $N$ for all $n \in \N$.  Then $x^n_{N+1} = y^n_{N+1} = \infty$ for all $n \in \N$, and taking limits shows $x_{N+1} = y_{N+1} = \infty$.  Also, since $x^n_i \neq \infty$ and $y^n_i \neq \infty$ for all $1 \leq i \leq N$, and since $Q(x^n) = Q(y^n)$ for all $n \in \N$, we have $x^n_i = y^n_i$ for all $1 \leq i \leq N$.  Thus $x_i = \lim_{n \to \infty} x^n_i = \lim_{n \to \infty} y^n_i = y_i$ for all $1 \leq i \leq N$.  Hence $Q(x) = Q(y)$ and $(x,y) \in G$.

The second case is that $\{ L(x^n) \}_{n=1}^\infty$ is not bounded.  By passing to a subsequence, we may assume that $\lim_{n \to \infty}  L(x^n) = \infty$.  Choose $i \in \N$.  Since $L(x^n) > i$ eventually, the fact that $Q(x^n) = Q(y^n)$ implies that $x^n_i = y^n_i$ for large enough $n$.  Thus $ \lim_{n \to \infty} x^n_i = \lim_{n \to \infty} y^n_i = y_i$.  Hence $x_i = y_i$ for all $i \in \N$, and $x=y$.  Hence $Q(x) = Q(y)$ and $(x,y) \in G$.
\end{proof}

\subsection{A Basis for the Topology on $\Sigma_\A$} \label{basis-top-subsec}

Since $\Sigma_\A$ is defined as a quotient space ---and quotient topologies are notoriously difficult to work with --- we shall exhibit a basis for $\Sigma_\A$ that will be convenient in many applications.  The basis we give is in terms of ``generalized cylinder sets" and it generalizes the topology one encounters in shift spaces over finite alphabets.  Using this basis we derive a characterization of sequential convergence in $\Sigma_\A$ (see Corollary~\ref{convergence-in-full-shift-cor}.)

\begin{definition}
If $x \in \Sfin$ and $y \in \Sigma_\A$ we define the \emph{concatenation} of $x$ and $y$ to be the sequence $xy \in \Sigma_\A$ obtained by listing the entries of $x$ followed by the entries of $y$.  We interpret $x\0=x$  for all $x \in \Sfin$ and $\0 y = y$ for all $y \in \Sigma_\A$.  Note that $l(xy) = l(x) + l(y)$ for all $x \in \Sfin$ and $y \in \Sigma_\A$.
\end{definition}

\begin{definition}
If $x \in \Sfin$, we define the \emph{cylinder set} of $x$ to be the set $$Z(x) := \{ xy : y \in \Sigma_\A \}.$$  Note that we always have $x \in Z(x)$ (simply take $y = \0$), and if $x,y \in \Sfin$ the following relation is satisfied:
\begin{equation} \label{cylinder-intersect-eq}
Z(x) \cap Z(y) = \begin{cases} Z(y) & \text{ if $y = xz$ for some $z \in \Sfin$} \\ 
Z(x) & \text{ if $x = yz$ for some $z \in \Sfin$} \\ 
\emptyset & \text{ otherwise.} \\ 
\end{cases}
\end{equation}
In addition, we have $Z(\0) = \Sigma_\A$.
\end{definition}

\begin{definition}
If $x \in \Sfin$ and $F \subseteq \A$ is a finite subset, we define the \emph{generalized cylinder set} of the pair $(x, F)$ to be the set $$Z(x,F) := Z(x) \setminus \bigcup_{e \in F} Z(xe).$$  Note that if $F = \emptyset$, then $Z(x,F) = Z(x)$ is a cylinder set.  Thus every cylinder set is also a generalized cylinder set.
\end{definition}

\begin{lemma} \label{cylinder-compopen-lem}
If $x \in \Sfin$ and $F \subseteq \A$ is a finite subset, the generalized cylinder set $Z(x,F)$ is a compact open subset of $\Sigma_\A$.
\end{lemma}

\begin{proof}
Let us first prove that any generalized cylinder set $Z(x,F)$ is open. Write $x = x_1 \ldots x_n$.  Since 
\begin{equation*}
Q^{-1} (Z(x,F)) = \{ x_1 \} \times \ldots \times \{ x_n \} \times (\A_\infty \setminus F) \times \A_\infty \times \A_\infty \times \ldots
\end{equation*} 
is open in $X_\A$, it follows that $Z(x,F)$ is open.

Next we shall show that every cylinder set $Z(x)$ is closed.  Suppose $y \notin Z(x)$ and $l(x)=n$.  Then either $y = x_1 \ldots x_k$ for $k < n$, or $y = x_1 \ldots x_k y_{k+1} \ldots$ with $y_{k+1} \neq x_{k+1}$ for some $k < n$.  In the first case, the generalized cylinder set $Z(x_1 \ldots x_k, \{ x_{k+1} \})$ is an open set with $y \in Z(x_1 \ldots x_k, \{ x_{k+1} \})$ and $Z(x_1 \ldots x_k, \{ x_{k+1} \}) \cap Z(x) = \emptyset$.  In the second case, the cylinder set $Z(x_1 \ldots x_k y_{k+1})$ is open with $y \in Z(x_1 \ldots x_k y_{k+1})$ and $Z(x_1 \ldots x_k y_{k+1}) \cap Z(x) = \emptyset$.  Hence $Z(x)$ is closed.  

Because every cylinder set is clopen, any generalized cylinder set
$$Z(x,F) := Z(x) \setminus \bigcup_{y \in F} Z(xy) = Z(x) \cap \left( \bigcup_{y \in F} Z(xy) \right)^c$$
is an intersection of closed sets and hence a closed set.  Since $Z(x,F)$ is a closed subset of the compact set $\Sigma_\A$, it follows that  $Z(x,F)$ is compact.  Hence any generalized cylinder set  $Z(x,F)$ is compact and open.
\end{proof}

Next we shall exhibit a basis for the topology on $\Sigma_\A$.  To do so, we will find it convenient to embed $\Sigma_\A$ into the space $2^{\Sfin} = \{ 0, 1 \}^{\Sfin}$.  Throughout, we consider $\{ 0, 1 \}^{\Sfin}$ as a topological space with the product topology.

\begin{definition}
We define a function $\alpha : \Sigma_\A \to \{ 0, 1 \}^{\Sfin}$ by $$\alpha (x) (y) = \begin{cases} 1 & \text{ if $x \in Z(y)$} \\ 0 & \text{ otherwise.} \end{cases}$$ 
\end{definition}

\begin{remark}
We may think of $\{ 0, 1 \}^{\Sfin}$ as the space of all subsets of $\Sfin$.  The map $\alpha : \Sigma_\A \to \{ 0, 1 \}^{\Sfin}$ then sends any element $x \in \Sigma_\A$ to the set of all the finite initial subsequences of $x$.
\end{remark}

\begin{definition}
If $F, G \subseteq \Sfin$ are disjoint finite subsets of $\Sfin$, we define a subset $N(F,G) \subseteq \{ 0, 1 \}^{\Sfin}$ by
\begin{equation*}
N(F,G) = \prod_{x \in \Sfin} N(F,G) (x),
\end{equation*}
where
$$N(F,G) (x) := \begin{cases}
\{ 1 \} & \text{ if $x \in F$} \\
\{ 0 \} & \text{ if $x \in G$} \\
\{ 0, 1 \} & \text{ otherwise.} \\
\end{cases}$$
We see that $\{ N(F,G) : \text{$F, G \subseteq \Sfin$ are disjoint finite subsets of $\Sfin$} \}$ is a basis for the topology on $\{ 0, 1 \}^{\Sfin}$.
\end{definition}

\begin{lemma}[cf.~Proposition~2.1.1 of \cite{Webster-thesis} and Theorem~2.1 of \cite{Web}]  \label{alpha-inv-lem}
If $F, G \subseteq \Sfin$ are disjoint finite subsets of $\Sfin$, then $$\alpha^{-1}( N(F,G)) = \left( \bigcap_{x \in F} Z(x) \right) \setminus \left( \bigcup_{y \in G} Z(y) \right).$$
\end{lemma}

\begin{proof}
If $z \in \Sigma_\A$, then 
\begin{align*}
z \in \alpha^{-1}( N(F,G)) &\Longleftrightarrow \alpha(z) \in N(F,G) \\
&\Longleftrightarrow \alpha(z)(x) = \begin{cases} 1 & \text{ if $x \in F$} \\ 0 & \text{ if $x \in G$} \end{cases} \\
&\Longleftrightarrow \text{$z \in Z(x)$ for all $x \in F$ and $z \notin Z(y)$ for all $y \in G$} \\
&\Longleftrightarrow z \in  \left( \bigcap_{x \in F} Z(x) \right) \setminus \left( \bigcup_{y \in G} Z(y) \right).
\end{align*}
\end{proof}

\begin{proposition} \label{alpha-embedding-prop}
The function $\alpha : \Sigma_\A \to \{ 0, 1 \}^{\Sfin}$ is an embedding; that is, $\alpha$ is a homeomorphism onto its image.
\end{proposition}

\begin{proof}
Let us first show that $\alpha$ is injective.  Suppose that $x,y \in \Sigma_\A$ and $\alpha(x) = \alpha(y)$.  Write $x=x_1x_2 \ldots$ and $y=y_1y_2\ldots$.  For every $n$ we have $\alpha(y)(x_1\ldots x_n) = \alpha(x)(x_1 \ldots x_n) = 1$, so that $y \in Z(x_1\ldots x_n)$ and $y_1 \ldots y_n = x_1 \ldots x_n$.  Since this holds for all $n$, we have $x=y$.

Next we shall show that $\alpha$ is continuous.  Since the collection of $N(F,G)$, where $F$ and $G$ range over all disjoint finite subsets of $\Sfin$, forms a basis for $\{ 0, 1 \}^{\Sfin}$, it suffices to show that $\alpha^{-1}(N(F,G))$ is open.  However, Lemma~\ref{alpha-inv-lem} shows that $\alpha^{-1}( N(F,G)) = \left( \bigcap_{x \in F} Z(x) \right) \setminus \left( \bigcup_{y \in G} Z(y) \right)$, and since the cylinder sets are clopen by Lemma~\ref{cylinder-compopen-lem}, it follows that the set $\left( \bigcap_{x \in F} Z(x) \right) \setminus \left( \bigcup_{y \in G} Z(y) \right)$ is open.  Hence $\alpha$ is continuous. 

Because $\Sigma_\A$ is compact, and $\alpha : \Sigma_\A \to \alpha(\Sigma_\A)$ is a continuous bijection, it follows from elementary point-set topology that $\alpha$ is a homeomorphism onto its image.
\end{proof}

Our proof of the following theorem relies on Lemma~\ref{cylinder-compopen-lem} and techniques similar to those used by Webster in the proof of \cite[Proposition~2.1.1]{Webster-thesis} and the proof of \cite[Theorem~2.1]{Web}.

\begin{theorem} \label{basis-for-full-shift-thm}
The collection of generalized cylinder sets
$$\{ Z(x,F) :  \text{ $x \in \Sfin$ and $F \subseteq \A$ is a finite subset} \}$$
is a basis for the topology of $\Sigma_\A$ consisting of compact open subsets.  In addition, if $x \in \Sigma_\A$ and $l(x) = \infty$, then a neighborhood base for $x$ is given by
$$\{ Z(x_1 \ldots x_n) : n \in \N \},$$
and if $x \in \Sigma_\A$ and $l(x) < \infty$, then a neighborhood base for $x$ is given by
$$\{ Z(x,F) : \text{$F$ is a finite subset of $\A$} \}.$$
\end{theorem}

\begin{proof}
It follows from Lemma~\ref{cylinder-compopen-lem} that the generalized cylinder sets $Z(x,F)$ are compact open subsets of $\Sigma_\A$.  In addition, Proposition~\ref{alpha-embedding-prop} shows that $\alpha : \Sigma_\A \to \{ 0, 1 \}^{\Sfin}$ is an embedding.  Since $$\{ N(F,G) : \text{$F, G \subseteq \Sfin$ are disjoint finite subsets of $\Sfin$} \}$$ is a basis for the topology on $\{ 0, 1 \}^{\Sfin}$, it follows that $$\{ \alpha^{-1}(N(F,G)) : \text{$F, G \subseteq \Sfin$ are disjoint finite subsets of $\Sfin$} \}$$ is a basis for the topology on $\Sigma_\A$.  Thus it suffices to show that for any $N(F,G)$ and any $z \in \alpha^{-1}(N(F,G))$ there exists a generalized cylinder set $Z(x,F')$ such that $z \in Z(x,F') \subseteq \alpha^{-1}(N(F,G))$.  We shall accomplish this in a few steps.

First, we shall show $\alpha^{-1}(N(F,G))$ can be written in a nicer form than that shown in Lemma~\ref{alpha-inv-lem}.  Given disjoint finite subsets $F,G \subseteq \Sfin$, we have 
$$\alpha^{-1}( N(F,G)) = \left( \bigcap_{x \in F} Z(x) \right) \setminus \left( \bigcup_{y \in G} Z(y) \right).$$
If $\alpha^{-1}( N(F,G)) \neq \emptyset$, then $\bigcap_{x \in F} Z(x) \neq \emptyset$.  It follows from \eqref{cylinder-intersect-eq} that $\bigcap_{x \in F} Z(x) = Z(w)$ for some $w \in F$.  Thus 
$$\alpha^{-1}( N(F,G)) = Z(w) \setminus \left( \bigcup_{y \in G} Z(y) \right) = Z(w) \setminus \left( \bigcup_{y \in G \cap Z(w)} Z(y) \right).$$
In addition, if we let $G' := \{ u \in \Sfin : wu \in G \cap Z(w)\}$, then $\bigcup_{y \in G \cap Z(w)} Z(y) = \bigcup_{u \in G'} Z(wu)$ so that 
$$\alpha^{-1}( N(F,G)) = Z(w) \setminus \left( \bigcup_{u \in G'} Z(wu) \right).$$

Next, let $z \in Z(w) \setminus \left( \bigcup_{u \in G'} Z(wu) \right)$.  We wish to find $x \in \Sfin$ and a finite subset $F' \subseteq \A$ such that $z \in Z(x,F') \subseteq Z(w) \setminus \left( \bigcup_{u \in G'} Z(wu) \right)$.  We consider two cases: $l(z) = \infty$ and $l(z) < \infty$.

If $l(z) = \infty$, let $N = \max \{ l(w u) : u \in G' \}$ if $G' \neq \emptyset$ or $N = l(w)$ if $G' = \emptyset$.  Define $x := z_1 \ldots z_N$ and $F' := \emptyset$.  Then $z \in Z(x, F')$, and since any element in $Z(x,F')$ has $x$, and hence also $w$, as its initial segment, we have 
$Z(x,F') \subseteq Z(w)$.  Furthermore, any element of $Z(x,F')$ has $x = z_1 \ldots z_N$ as an initial segment, and since $N \geq l(wu)$ for all $u \in G'$, and $z \notin Z(wu)$, this element does not have $wu$ as an initial segment.  Thus $Z(x,F') \subseteq  \left(  \bigcup_{u \in G'} Z(wu) \right)^c$.  It follows that
$$z \in Z(x,F') \subseteq Z(w) \setminus \left( \bigcup_{u \in G'} Z(wu) \right)$$
as desired.  This also shows that $\{ Z(z_1 \ldots z_n) : n \in \N \}$ is a neighborhood base of $z$. 

If $l(z) < \infty$, let $x := z$ and $F' := \{ (w u)_{l(z)+1} : u \in G' \text{ and } l(wu) > l(z) \}$.  Then $z \in Z(x,F')$ since $x=z$.  To see that $Z(x,F') \subseteq Z(w) \setminus \left( \bigcup_{u \in G'} Z(wu) \right)$, fix $\alpha \in Z(x,F')$.  Write $z = w z'$ for $z' \in \Sfin$, and $\alpha = x \alpha'$ for $\alpha' \in \Sigma_\A$.  Then $\alpha = x\alpha' = z\alpha'=wz'\alpha' \in Z(w)$.  Also, fix $u \in G'$.  If $l(wu) \leq l(z)$, then $l(u) \leq l(z')$, and since $z' \notin Z(u)$, we have $z' \alpha' \notin Z(u)$, and $w z' \alpha' \notin Z(wu)$, and $\alpha \notin Z(wu)$.  On the other hand, if $l(wu) > l(z)$, then since $\alpha \in Z(x,F')$ we have $\alpha_{1}' \notin F'$, and $\alpha_{l(z)+1} = (wz'\alpha')_{l(z)+1} = (z\alpha')_{l(z)+1} = \alpha'_1 \neq (wu)_{l(z)+1}$.  Hence $\alpha \notin Z(wu)$.   It follows that 
 $Z(x,F') \subseteq \left( \bigcup_{u \in G'} Z(wu) \right)^c$.  Thus
$$z \in Z(x,F') \subseteq Z(w) \setminus \left( \bigcup_{u \in G'} Z(wu) \right)$$
as desired.  This also shows that $\{ Z(z,F) : \text{$F$ is a finite subset of $\A$} \}$ is a neighborhood base of $z$.
\end{proof}

\begin{remark}
A basis for a similar topology was described in \cite[Corollary~2.4]{PW}, however, as pointed out in \cite[p.12]{Webster-thesis} there is a minor oversight in \cite[Corollary~2.4]{PW} and it fails to include some of the necessary basis elements.
\end{remark}

\begin{corollary} \label{convergence-in-full-shift-cor}
Let $\{ x^n \}_{n=1}^\infty$ be a sequence of elements in $\Sigma_\A$ and write $x^n = x^n_1 x^n_2 \ldots$ with $x^n_i \in \A$ for all $i \in \N$.  Also let $x = x_1 x_2 \ldots \in \Sigma_\A$.
\begin{itemize}
\item[(a)] If $l(x) = \infty$, then  $\lim_{n \to \infty} x^n = x$ with respect to the topology on $\Sigma_\A$ if and only if for every $M \in \N$ there exists $N \in \N$ such that $n > N$ implies $x^n_i = x_i$ for all $1 \leq i \leq M$.
 \item[(b)] If $l(x) < \infty$, then $\lim_{n \to \infty} x^n = x$ with respect to the topology on $\Sigma_\A$ if and only if for every finite subset $F \subseteq \A$ there exists $N \in \N$ such that $n > N$ implies $l(x^n) \geq l(x)$, $x^n_{l(x)+1} \notin F$, and $x_i^n = x_i$ for all $1 \leq i \leq l(x)$. (Note: If $l(x^n) = l(x)$ we consider the condition $x^n_{l(x)+1} \notin F$ to be vacuously satisfied.)
\end{itemize}
\end{corollary}

\begin{proof}
This follows from the description of the neighborhood bases of points described in Theorem~\ref{basis-for-full-shift-thm}.
\end{proof}

\begin{corollary} \label{countability-cor}
The following are equivalent:
\begin{itemize}
\item[(i)] The set $\A$ is countable.
\item[(ii)] The space $\Sigma_\A$ is second countable.
\item[(iii)] The space $\Sigma_\A$ is first countable.
\end{itemize}
\end{corollary}

\begin{proof}
If (i) holds, then $\Sfin$ is countable and the collection of finite subsets of $\A$ is countable, and hence the collection $$\{ Z(x,F) :  \text{ $x \in \Sfin$ and $F \subseteq \A$ is a finite subset} \}$$ of generalized cylinder sets is countable.  Thus $\Sigma_\A$ is second countable and (ii) holds.  We have (ii) implies (iii) trivially.

If (iii) holds, choose $x \in \Sfin$ and choose a countable neighborhood base $\{ U_i \}_{i \in \N}$ for $x$.  For each $i \in \N$ choose a finite subset $F_i \subseteq \A$ such that $Z(x,F_i) \subseteq U_i$.  Since $\Sigma_\A$ is Hausdorff, we have $\bigcap_{i=1}^\infty U_i = \{ x \}$.  Thus, $x \in \bigcap_{i=1}^\infty Z(x, F_i) \subseteq \bigcap_{i=1}^\infty U_i = \{ x \}$, so that $\bigcap_{i=1}^\infty Z(x, F_i) = \{ x \}$ and $\bigcup_{i=1}^\infty F_i = \A$.  Since $\A$ is the countable union of finite sets, $\A$ is countable and (i) holds.
\end{proof}

\begin{remark}
When $\A$ is countable, Corollary~\ref{countability-cor} shows that $\Sigma_\A$ is first (and second) countable.  In this case, sequences suffice to determine the topology, and all topological information can be obtained using the sequential convergence described in Corollary~\ref{convergence-in-full-shift-cor}.
\end{remark}

\begin{remark}
Even though Corollary~\ref{countability-cor} shows that the space $\Sigma_\A$ is first countable if and only if the alphabet $\A$ is countable, for any $\A$ and any $x \in \Sinf$ the collection $\{Z(x_1 \ldots x_n) : n \in \N \}$ is a countable neighborhood base of $x$.
\end{remark}

\subsection{A Family of Metrics on $\Sigma_\A$ when $\A$ is Countable} \label{metrics-sec}

When $\A$ is countable, Corollary~\ref{countability-cor} shows that $\Sigma_\A$ is a second countable compact Hausdorff space and hence is metrizable.  Assuming $\A$ is countable, we describe a family of metrics on $\Sigma_{\A}$ that induce the topology.  We do so by embedding $\Sigma_\A$ into a metric space and then using the embedding to ``pull back" the metric to a metric on $\Sigma_\A$.

\begin{example}
We use the embedding $\alpha : \Sigma_\A \to \{ 0, 1 \}^{\Sfin}$ described in Proposition~\ref{alpha-embedding-prop}.  If $\A$ is countable, then the set of finite sequences $\Sfin$ is countable.  Thus we may list the elements of $\Sfin$ as $\Sfin = \{ p_1, p_2, p_3, \ldots \}$, order $\{ 0, 1 \}^{\Sfin}$ as $$\{ 0, 1 \}^{\Sfin} = \{0, 1 \}_{p_1} \times \{0,1\}_{p_2} \times \ldots,$$
and define a metric $d_\textrm{fin}$ on $\{ 0, 1 \}^{\Sfin}$ by 
$$d_\textrm{fin} (\mu, \nu) := \begin{cases} 1/2^i & \text{$i \in \N$ is the smallest value such that $\mu(i) \neq \nu(i)$} \\
0 & \text{if $\mu(i) = \nu (i)$ for all $i \in \N$}.
\end{cases}$$
The metric $d_\textrm{fin}$ induces the product topology on $\{ 0, 1 \}^{\Sfin}$, and hence the topology on $\Sigma_\A$ is induced by the metric $d_\A$ on $\Sigma_\A$ defined by $d_\A (x,y) := d_\textrm{fin}( \alpha(x), \alpha(y))$.  Note that for $x,y \in \Sigma_\A$, we have
$$d_\A (x,y) := \begin{cases} 1/2^i & \text{$i \in \N$ is the smallest value such that $p_i$ is an initial} \\  & \text{ \ \ \ subsequence of one of $x$ or $y$ but not the other} \\
0 & \text{if $x=y$}.
\end{cases}$$ 
The metric $d_\A$ depends on the order we choose for $\Sfin = \{ p_1, p_2, p_3, \ldots \}$.
\end{example}

\subsection{The Shift Map}  \label{shift-map-subsec}

We next consider the ``shift map" on $\Sigma_\A$.

\begin{definition} \label{shift-map-def}
The \emph{shift map} is the function $\sigma : \Sigma_\A \to \Sigma_\A$ defined by $$\sigma(x) =  \begin{cases} x_2 x_3 \ldots & \text{ if $x = x_1 x_2 \ldots \in \A^\N$} \\ x_2 \ldots x_n & \text{ if $x = x_1 \ldots x_n \in \bigcup_{k=2}^\infty \A^k$} \\ \0 & \text{ if $x \in \A^1 \cup \{ \0 \}$.} \end{cases}$$
\end{definition}

Note that if $l(x) = \infty$, then $l (\sigma(x)) = \infty$, if $l(x) \in \N$, then $l(\sigma(x)) = l(x)-1$, and if $l(x) = 0$, then $l(\sigma(x)) = 0$.  Also note that if $x \in \Sigma_\A \setminus \{ \0 \}$, then $\sigma (x)_i = x_{i+1}$ for $1 \leq i < l(x)$.

\begin{proposition} \label{shift-cts-all-but-empty-word}
Let $\A$ be an infinite alphabet.  The shift map $\sigma :  \Sigma_\A \to \Sigma_\A$ is continuous at all points in $\Sigma_\A \setminus \{ \0 \}$ and discontinuous at the point $\0$.  In addition, if $x \in \Sigma_\A \setminus \{ 0 \}$, then there exists an open set $U \subseteq \Sigma_\A \setminus \{ \0 \}$ such that $x \in U$, $\sigma(U)$ is an open subset of $\Sigma_\A$, and $\sigma|_U : U \to \sigma(U)$ is a homeomorphism.
\end{proposition}

\begin{proof}
Let $x \in \Sigma_\A \setminus \{ \0 \}$, and let $V \subseteq \Sigma_\A$ be an open set with $\sigma(x) \in V$.  Since $x \neq \0$, there exists $a \in \A$ such that $x = a \sigma(x)$.  By Theorem~\ref{basis-for-full-shift-thm} there exists a compact open neighborhood $Z(y, F)$ of $\Sigma_\A$ with $\sigma(x) \in Z(y,F) \subseteq V$.  If we let $U := Z(ay, F)$, then $U$ is an compact open subset of $\Sigma_\A$, $x \in U$, and $\sigma(U) = Z(y,F) \subseteq V$.  Hence $\sigma$ is continuous at $x$.  

In addition, since $\sigma|_U : U \to \sigma(U)$ is bijective with inverse $z \mapsto az$, we see that  $\sigma|_U : U \to \sigma (U)$ is a continuous bijection from the compact open set $U$ onto the open set $\sigma(U)$, and hence $\sigma|_U : U \to \sigma (U)$ is a homeomorphism. 

To see that $\sigma$ is discontinuous at $\0$, choose a sequence of distinct elements $a_1, a_2, \ldots \in \A$.  For each $n \in \N$, define a sequence $\{ x^n \}_{n=1}^\infty$ defined by $x^n := a_n a_1 a_1 a_1 \ldots$.  Then $\lim_{n \to \infty} x^n = \0$, and we see $\sigma( \lim_{n \to \infty} x^n) = \sigma (\0) = \0$, while $\lim_{n \to \infty} \sigma (x^n) = \lim_{n \to \infty} a_1 a_1 \ldots = a_1 a_1 \ldots$.  Hence $\sigma( \lim_{n \to \infty} x^n) \neq \lim_{n \to \infty} \sigma (x^n)$, and $\sigma$ is not continuous at $\0$.
\end{proof}

\begin{remark}
Recall that in the case of a finite alphabet, the full shift $\Sigma_\A$ consists of infinite sequences of letters from $\A$, and in particular $\Sigma_\A$ does not contain the empty sequence $\0$, and the shift map $\sigma : \Sigma_\A \to \Sigma_\A$ is continuous at all points.  When $\A$ is infinite, $\Sigma_\A$ contains the empty sequence $\0$, and Proposition~\ref{shift-cts-all-but-empty-word} shows that the shift map $\sigma : \Sigma_\A \to \Sigma_\A$ has a single discontinuity at $\0$.  This lack of continuity will not cause us any difficulty, since nothing we do in the sequel will require continuity of the shift map.
\end{remark}

The results of this section allow us to make the following definition.

\begin{definition}
If $\A$ is an infinite alphabet, we define the \emph{one-sided full shift} to be the pair $(\Sigma_\A, \sigma)$ where $\Sigma_\A$ is the topological space from Definition~\ref{Sigma-X-def} and $\sigma : \Sigma_\A \to \Sigma_\A$ is the map from Definition~\ref{shift-map-def}.  When it is clear from context that we are discussing one-sided shifts, we shall often refer to $(\Sigma_\A, \sigma)$ as simply the \emph{full shift} on the alphabet $\A$.  In addition, as in the classical case we engage in some standard sloppiness and often refer to the space $\Sigma_\A$ as the \emph{full shift} with the understanding that the map $\sigma$ is attached to it.
\end{definition}

\begin{remark}
We assumed throughout this past section that $\A$ is infinite, but when $\A$ is finite we can repeat our construction.  
In this case $\A$ with the discrete topology is compact, and the minimal compactification of $\A$ is $\A$ itself, so that $\A_\infty = \A$.  We perform our construction as above, and all statements about the element $\infty$ are then vacuous.  Thus $X_\A = \A \times \A \times \A \times \ldots = \A^\N$ with the product topology, and the quotient map $Q : X_\A \to \A^\N \cup \bigcup_{k=0}^\infty \A^k$ is the inclusion map $\A^\N \hookrightarrow \A^\N \cup \bigcup_{k=0}^\infty \A^k$.  Thus the image of $Q$ is simply the space $\A^\N$ and the quotient topology induced by $Q$ is the product topology on $\A^\N$.  Hence when $\A$ is finite, we recover the usual definition of the full shift as $\Sigma_\A := \A^\N$ with the product topology, and every sequence in the full shift has infinite length.  We also observe that in this case the collection of cylinder sets $$\{ Z(x_1\ldots x_n) : \text{$n \in \N$ and $x_i \in \A$ for $1 \leq i \leq n$} \}$$ forms a basis for the topology on $\Sigma_\A$.
\end{remark}

\section{Shift Spaces over Infinite Alphabets} \label{shift-spaces-sec}

Having defined the full shift over an arbitrary alphabet in the previous section, we now use it to define shift spaces as subspaces of the full shift having certain properties.  In addition to requiring a shift space to be closed and invariant under the shift map, we will also require that it satisfies what we call the ``infinite-extension property".

\begin{definition} \label{inf-ext-def}
If $\A$ is an alphabet and $X \subseteq \Sigma_\A$, we say $X$ has the \emph{infinite-extension property} if for all $x \in X$ with $l(x) < \infty$, there are infinitely many $a \in \A$ such that $Z(xa) \cap X \neq \emptyset$.
\end{definition}

\begin{remark}
Note that $X$ has the infinite-extension property if and only if whenever $x \in X$ and $l(x) < \infty$, then the set $$\{ a \in \A : x a y \in X \text{ for some $y \in \Sigma_\A$} \}$$
is infinite.
\end{remark}

\begin{definition} \label{shift-space-def}
Let $\A$ be an alphabet, and $(\Sigma_\A, \sigma)$ be the full shift over $\A$.  A \emph{shift space} over $\A$ is defined to be a subset $X \subseteq \Sigma_\A$ satisfying the following three properties:
\begin{itemize}
\item[(i)] $X$ is a closed subset of $\Sigma_\A$.
\item[(ii)] $\sigma (X) \subseteq X$.
\item[(iii)] $X$ has the infinite-extension property.
\end{itemize}
For any shift space $X$ we define $\Xinf := X \cap \Sinf$ and $\Xfin := X \cap \Sfin$.
\end{definition}

\begin{remark}
Since $\Sigma_\A$ is compact, Property~(i) implies that any shift space is compact.  In addition, Property~(ii) implies that $\sigma : \Sigma_\A \to \Sigma_\A$ restricts to a map $\sigma|_X : X \to X$.  Thus we will often attach the map $\sigma|_X$ to $X$ and refer to the pair $(X, \sigma|_X)$ as a \emph{shift space}.   Note that our definition allows the empty set $X = \emptyset$ as a shift space.  However, Property~(iii) shows that if $X \neq \emptyset$, then $\Xinf \neq \emptyset$, so that nonempty shift spaces will always have sequences of infinite length (see Proposition~\ref{inf-ext-in-shift-prop}).
\end{remark}

\begin{remark} \label{recover-classical-rem}
If $\A$ is finite, then $\Sigma_\A$ contains no finite sequences and any subset of $\Sigma_\A$ vacuously satisfies the infinite-extension property.  Consequently, when $\A$ is finite a subset $X \subseteq \Sinf$ is a shift space if and only  if $X$ is closed and $\sigma (X) \subseteq X$.  Thus when $\A$ is finite we recover the ``classical theory" of shift spaces.  We also observe that if $X$ is a shift space over a finite alphabet, then $\Xinf = X$ and $\Xfin = \emptyset$.
\end{remark}

\begin{remark}
Any shift space is a topological space with the subspace topology generated by the basis elements
$$Z_X(\alpha, F) := Z(\alpha, F) \cap X = \{ \alpha \beta : \alpha \beta \in X \text{ and } \beta_1 \notin F \}$$
for all $\alpha \in \Sfin$ and all finite subsets $F \subseteq \A$.  When we are working with a given shift $X$, we shall often omit the subscript $X$ and simply write $Z(\alpha, F)$ for the intersection of the generalized cylinder set with $X$.
\end{remark}

The following proposition shows that the infinite-extension property implies that a finite sequence in a shift space may be extended to an infinite sequence in the shift space with infinitely many choices of the first symbol.

\begin{proposition} \label{inf-ext-in-shift-prop}
If $X$ is a shift space and $x \in \Xfin$, then there exists $y \in \Sinf$ such that $xy \in X$.  Moreover, if $F$ is a finite subset of $\A$, then $y$ may be chosen so that $y_1 \notin F$.
\end{proposition}

\begin{proof}
If $x \in \Xfin$, then by the infinite-extension property of $X$ there exists $a_1 \in \A \setminus F$ and $y^1 \in \Sigma_\A$ with $xa_1y^1 \in X$.  If $xa_1y^1 \in X$ is infinite, we are done.  If not, we do the same process to $xa_1y^1$ and continue recursively, at each step either finding an infinite-extension of $x$ that is in $X$ or finding an element
$$z^n := x a_1 y^1 a_2 y^2 \ldots a_n y^n \in X$$
of finite length.  We see that $\{ z^n \}_{n=1}^\infty$ is a sequence in $X$ with $\lim_{n \to \infty} z^n = xa_1 y^1a_2y^2 \ldots \in X$.  Moreover, $xa_1 y^1a_2y^2 \ldots$ is an infinite sequence.
\end{proof}

The following proposition shows that a shift space $X$ is determined by the subset $\Xinf$.

\begin{proposition} \label{Xinf-closure-is-X-prop}
If $X \subseteq \Sigma_\A$ is a shift space, then $\Xinf$ is dense in $X$.
\end{proposition}

\begin{proof}
Suppose that $x \in X$ with $l(x) < \infty$.  By Proposition~\ref{basis-for-full-shift-thm} the collection of $Z(x,F)$ such that $F$ is a finite subset of $\A$ is a neighborhood base of $x$.  By Proposition~\ref{inf-ext-in-shift-prop} there exists $y \in \Sinf$ such that $y_1 \notin F$ and $xy \in X$.  Hence $xy \in Z(x,F) \cap \Xinf$, and $x$ is a limit point of $\Xinf$.
\end{proof}

\begin{corollary} \label{Xinf-determines-X-cor}
If $X \subseteq \Sigma_\A$ and $Y \subseteq \Sigma_\A$ are shift spaces over $\A$, then $X=Y$ if and only if $\Xinf = Y^\textnormal{inf}$.
\end{corollary}

Having defined shift spaces, our next order of business is to show that, as in the classical case, we can describe any shift space in terms of its ``forbidden blocks".

\begin{definition}
We will use the term \emph{block} as another name for the elements of $\Sfin := \bigcup_{k=0}^\infty \A^k$, with the \emph{empty block} being our empty sequence $\0$.  If $x \in \Sigma_\A$, a \emph{subblock} of $x$ is an element $u \in \Sfin$ such that $x = yuz$ for some $y \in \Sfin$ and some $z \in \Sigma_\A$.  By convention, the empty block $\0$ is a subblock of every element of $\Sigma_\A$.
\end{definition}

\begin{definition}
If $\F \subseteq \Sfin$, we define
\begin{align*}
\Xinf_\F &:= \{ x \in \Sinf : \text{ no subblock of $x$ is in $\F$} \} \\
\Xfin_\F &:= \{ x \in \Sfin : \text{ there are infinitely many $a \in \A$ for which}\\
& \qquad \qquad \qquad \qquad \text{there exists $y \in \Sinf$ such that $xay \in \Xinf_{\F}$}\} \\
X_\F &:= \Xinf_\F \cup \Xfin_\F.
\end{align*}
\end{definition}

\begin{remark}
Note that if $\0 \in \F$, then $X_\F = \emptyset$ is the empty shift space.  Hence one must have $\F \subseteq \Sfin \setminus \{ \0 \}$ to produce a nondegenerate shift space.
\end{remark}

\begin{remark}
If $\A$ is infinite and $\F = \emptyset$, then $X_\F = \Sigma_\A$ is the full shift.
\end{remark}

\begin{proposition} \label{XF-shift-space-prop}
If $\F \subseteq  \Sfin$, then $X_\F$ is a shift space.
\end{proposition}

\begin{proof}
First, we show that $X_\F$ is closed.  Suppose that we have a sequence $\{ x^n \}_{n=1}^\infty \subseteq X_\F$ and that $\lim_{n \to \infty} x^n = x \in \Sigma_\A$.  If $l(x) = \infty$, then by Corollary~\ref{convergence-in-full-shift-cor}  for every $M \in \N$ there exists $N \in \N$ such that $n > N$ implies that $x_i^n = x_i$ for all $1 \leq i \leq M$.  Hence $x_1 \ldots x_M = x^n_1 \ldots x^n_M$ for all $n > N$.  Since $x^n \in X_\F$, no subblock of $x^n$ is in $\F$.  Hence no subblock of $x_1 \ldots x_M$ is in $\F$, and since this holds for all $M \in \N$, it follows no subblock of $x$ is in $\F$, and hence $x \in \Xinf_\F \subseteq X_\F$.  

If $l(x) < \infty$, then let $F$ be any finite subset of $\A$.  By Corollary~\ref{convergence-in-full-shift-cor} there exists $n \in \N$ such that $l(x^n) \geq l(x)$, $x^n_{l(x)+1} \notin F$, and $x_i^n = x_i$ for all $1 \leq i \leq l(x)$.  Thus $x$ agrees with $x^n$ in the first $l(x)$ entries with $x^{n}_{l(x)+1} \notin F$.  Since $x^n$ is either in $\Xinf_\F$ or $\Xfin_\F$, we can find $a \notin F$ and $y \in \Xinf_{\F}$ such that $xay \in \Xinf_\F$.  Since this is true for any finite subset $F \subseteq \A$, there exist infinitely many $a \in \A$ with the property that there is $y \in \Xinf_\F$ such that $xay \in \Xinf_\F$.  Hence $x \in \Xfin_\F \subseteq X_\F$.  Thus $X_\F$ is closed.

Next, we observe that $\sigma (\Xinf_\F) \subseteq \Xinf_\F$ and $\sigma(\Xfin_\F) \subseteq \Xfin_\F$ so that $\sigma (X_\F) \subseteq X_\F$. Finally, we verify that $X_\F$ has the infinite-extension property.  This is an immediate consequence of the definition of $\Xfin_\F$: If $x \in X_\F$ and $l(x) < \infty$, then $x \in \Xfin_\F$ and by the definition of $\Xfin_\F$ there exist infinitely many $a \in \A$ for which there is an element $y \in \Sinf$ such that $xay \in \Xinf_\F \subseteq X_\F$.  Hence $X_\F$ has the infinite-extension property.
\end{proof}

\begin{definition}
Let $X \subseteq \Sigma_\A$.  We define the \emph{set of blocks of $X$} to be $$B(X) := \{ u \in \Sfin : \text{ $u$ is a subblock of some element of $X$} \}.$$  For $n \in \N \cup \{ 0 \}$ we define the \emph{set of $n$-blocks of $X$} to be $$B_n (X) := \{ u \in \A^n : \text{ $u$ is a subblock of some element of $X$} \}.$$ Note that $B_0(X) = \{ \0 \}$, and $B_1(X) \subseteq \A$ is the set of symbols that appear in the elements of $X$.  In addition, $B(X) = \bigcup_{n=0}^\infty B_n(X)$.
\end{definition}

\begin{theorem}
A subset $X \subseteq \Sigma_\A$ is a shift space if and only if $X = X_\F$ for some subset $\F \subseteq \Sfin$.
\end{theorem}

\begin{proof}
If $X = X_\F$ for some subset $\F \subseteq \Sfin$, then $X$ is a shift space by Proposition~\ref{XF-shift-space-prop}.

Conversely, let $X \subseteq \Sigma_\A$ be a shift space.  Define $\F := \Sfin \setminus B(X)$, so that $\F$ consists of those elements in $\Sfin$ that are not subblocks of elements of $X$.  We shall show that $X = X_\F$.  

Let $x \in X_\F$, and consider two cases.

\smallskip

\noindent \textsc{Case I:} $l(x) = \infty$.  

Then $x \in \Xinf_{\F}$, and no subblock of $x$ is in $\F$.  Thus for all $n \in \N$ we have that $x_1 \ldots x_n \notin \F$, and $x_1 \ldots x_n \in B(X)$.  Hence for all $n \in \N$, there exists $u^n \in \Sfin$ and $y^n \in \Sigma_\A$ such that $u^nx_1 \ldots x_n y^n \in X$.  Because $\sigma (X) \subseteq X$, it follows $x_1 \ldots x_n y^n = \sigma^{l(u^n)} (u^n x_{1} \ldots x_{n} y^n) \in X$.  Because $x = \lim_{n \to \infty} x_{1} \ldots x_{n} y^n$ and $X$ is closed, we have $x \in X$.

\smallskip 

\noindent \textsc{Case II:} $l(x) < \infty$.  

Then $x \in \Xfin_\F$, and there exists an infinite sequence of distinct elements $a_1, a_2, \ldots \in \A$ such that for each $n \in \N$ there exists $y^n \in \Sinf$ such that $xa_ny^n \in \Xinf_\F$.  By the argument of Case~I we have $x a_{n} y^{n} \in X$ for all $n \in \N$.  Since $X$ is closed, $x = \lim_{n \to \infty} xa_ny^n \in X$.

Thus we have shown that $X_\F \subseteq X$.  For the reverse inclusion, suppose that $x \in X$.  If $x \in \Xinf$, then $x$ has infinite length and no subblock of $x$ is in $\F$, so $x \in \Xinf \subseteq X_\F$.  If $x \in \Xfin$, then by the infinite-extension property of $X$, there exist an infinite sequence of distinct elements $a_1, a_2, \ldots \in \A$ such that for each $n \in \N$ there exists $y^n \in \Sigma_\A$ such that $xa_ny^n \in X$.  By Proposition~\ref{inf-ext-in-shift-prop} we may assume that each $y^n$ is infinite, and hence each $xa_ny^n$ is infinite.  Since each $xa_ny^n \in X$ is infinite, and no subblock of $xa_ny^n$ is in $\F$, we have $xa_ny^n \in \Xinf_\F$, and hence by the definition of $\Xfin_\F$, we have $x \in \Xfin_\F$.  Thus $X \subseteq X_\F$.
\end{proof}

We conclude this section by discussing shift spaces with certain finiteness restrictions on the allowed symbols.

\begin{definition}
Let $\A$ be an alphabet, and let  $X \subseteq \Sigma_\A$ be a shift space over $\A$.  We say that $X$ is \emph{finite-symbol} (or \emph{finite}) if $B_1(X)$ is finite, and we say $X$ is \emph{infinite-symbol} (or \emph{infinite}) otherwise.  We say that $X$ is \emph{row-finite} if for every $a \in \A$, the set $\{ b \in \A :  ab \in B(X) \}$ is finite.  
\end{definition}

\begin{remark}
Note that every finite-symbol shift space is row-finite.  Also note that if $\A$ is a finite set, then every shift space over $\A$ is finite-symbol.
\end{remark}

The following propositions give us several alternate ways to characterize finite-symbol and row-finite shift spaces.

\begin{proposition} \label{fin-char-prop}
Let $X \subseteq \Sigma_\A$ be a shift space.  Then the following are equivalent:
\begin{itemize}
\item[(i)] $X$ is finite-symbol.
\item[(ii)] $\Xfin = \emptyset$.
\item[(iii)] $X = \Xinf$.
\item[(iv)] $\Xinf$ is a closed subset of $\Sigma_\A$.
\item[(v)] $\0$ is not a limit point of $\Xinf$.
\item[(vi)]  $\0 \notin X$.
\end{itemize}
\end{proposition}

\begin{proof}
$(i) \Rightarrow (ii)$. Since $B_1(X)$ is finite, the infinite-extension property of $X$ implies that $\Xfin = \emptyset$.

$(ii) \Rightarrow (iii)$. Since $X = \Xinf \cup \Xfin$, the result follows.

$(iii) \Rightarrow (iv)$. If $X = \Xinf$, then Proposition~\ref{Xinf-closure-is-X-prop} implies that $\overline{\Xinf} = X = \Xinf$, so $\Xinf$ is closed.

$(iv) \Rightarrow (v)$. Since $\0 \notin \Xinf$ and $\Xinf$ is closed, $\0$ is not a limit point of $\Xinf$.

$(v) \Rightarrow (vi)$. Proposition~\ref{Xinf-closure-is-X-prop} implies that $\overline{\Xinf} = X$.  Thus if $\0$ is not a limit point of $\Xinf$, it follows that $\0 \notin X$.

$(vi) \Rightarrow (i)$.  Suppose $X$ is not finite-symbol.  Then there exists an infinite sequence $a_1, a_2, \ldots \in B_1(X)$ of distinct elements.  Hence, using Proposition~\ref{inf-ext-in-shift-prop} for each $n \in \N$ there exists $z^n \in \Sfin$ and $x^n \in \Sinf$ such that $z^na_nx^n \in X$.  Since $X$ is closed under the shift map $\sigma$, we have $a_n x^n = \sigma^{l(z^n)} (z^na_nx^n) \in X$.  Finally, since $X$ is closed and the $a_n$ are distinct, $\0 = \lim_{n \to \infty} a_nx^n \in X$.
\end{proof}

\begin{remark}
The astute reader may be concerned that Proposition~\ref{fin-char-prop} implies that the collection of infinite paths in a row-finite infinite graph is not a shift space.  This is true, and we will give a satisfactory  explanation in Section~\ref{analogues-of-SFT-sec}.  In particular, we shall see in Definition~\ref{edge-shift-closure-Einf-def} and Proposition~\ref{description-of-edge-shift-prop} that the edge shift of a graph is defined to be the closure of the collection of infinite paths, and thus the edge shift of a row-finite infinite path is equal to the infinite path space together with the empty sequence $\0$.  The next proposition shows that in a row-finite shift space the only possible finite sequence is $\0$.
\end{remark}

\begin{proposition} \label{rf-char-prop}
Let $X \subseteq \Sigma_\A$ be a shift space.  Then the following are equivalent:
\begin{itemize}
\item[(i)] $X$ is row-finite.
\item[(ii)] $\Xfin \subseteq \{ \0 \}$.
\item[(iii)] $\Xinf \cup \{ \0 \}$ is a closed subset of $\Sigma_\A$.
\item[(iv)] Either $X = \Xinf \cup \{ \0 \}$ or $X$ is finite-symbol.
\item[(v)] For each $a \in \A$, the set $\{ b \in \A : \text{there exists $x \in \Xinf$ with $abx \in X$} \}$ is finite.
\end{itemize}
\end{proposition}

\begin{proof}
$(i) \Rightarrow (ii)$. If $x \in \Xfin$, then by the infinite-extension property there are infinitely many $a \in \A$ such that $xa \in B(X)$.  However, since $X$ is a row-finite shift space, the only way this can occur is if $x = \0$.

$(ii) \Rightarrow (iii)$.  If $\Xfin \subseteq \{ \0 \}$, then either $\Xfin = \{ \0 \}$ or $\Xfin = \emptyset$.  If $\Xfin = \{ \0 \}$, then $\Xinf \cup \{ \0 \} = \Xinf \cup \Xfin = X$ is closed since $X$ is a shift space.  If $\Xfin = \emptyset$, then $X = \Xinf$ and $\Xinf \cup \{ \0 \} = X \cup \{ \0 \}$ is the union of two closed sets, and hence closed.

$(iii) \Rightarrow (iv)$. If $\Xinf \cup \{ \0 \}$ is closed, then $\overline{\Xinf}$ is equal to either $\Xinf$ or $\Xinf \cup \{ \0 \}$.  From Proposition~\ref{Xinf-closure-is-X-prop} we have that $\overline{\Xinf} = X$.  Thus either $X = \Xinf\cup \{ \0 \}$ or $X = \Xinf$.  In the latter case, $X$ is finite-symbol by Proposition~\ref{fin-char-prop}(iii).

$(iv) \Rightarrow (v)$.  If $(v)$ does not hold, then for some $a \in \A$, there is an infinite sequence $b_1, b_2, \ldots \in \A$ of distinct elements with the property that for each $n \in \N$ there exists $x^n \in \Xinf$ such that $a b_{n} x^n \in X$.  Then $b_n \in B(X)$ for all $n \in \N$, and $X$ is not finite-symbol.  In addition, since $X$ is a shift space and therefore closed, we have $a = \lim_{n \to \infty} ab_nx^n \in X$, which implies $X \neq \Xinf \cup \{ \0 \}$.  Hence $(iv)$ does not hold.

$(v) \Rightarrow (i)$.  If $(i)$ does not hold, then there exists $a \in \A$ and an infinite sequence $b_1, b_2, \ldots \in \A$ of distinct elements with the property that $ab_n \in B(X)$ for all $n \in \N$.  This implies that for each $n \in \N$ the block $ab_n$ is a subblock of an element of $X$.  Hence, using Proposition~\ref{inf-ext-in-shift-prop} there exists $z \in \Sfin$ and $x \in \Sinf$ such that $zab_nx \in X$.  Because $X$ is closed under the shift map $\sigma$, we have $ab_nx = \sigma^{l(z)} (zab_nx) \in X$.  Since the $b_n$ are distinct, this shows $(v)$ does not hold.
\end{proof}

Proposition~\ref{fin-char-prop} and Proposition~\ref{rf-char-prop} show that if we have a row-finite shift space, then the presence of the empty sequence $\0$ determines whether $X$ is finite-symbol or infinite-symbol.

\begin{corollary}
Let $X$ be a row-finite shift space.  Then either $X = \Xinf$ or $X = \Xinf \cup \{ \0 \}$.  If $X = \Xinf$, then $X$ is finite-symbol.  If $X = \Xinf \cup \{ \0 \}$, then $X$ is infinite-symbol.
\end{corollary}

\begin{remark}
Suppose that $\A$ is an alphabet, and that $X \subseteq \Sigma_\A$ is a symbol-finite shift space.  Then it follows from Proposition~\ref{fin-char-prop} that $\Xfin = \emptyset$.  Moreover, $X$ is then a subset of $\Sigma_{B_1(X)}$, the full shift over the finite alphabet $B_1(X)$.  As described in Remark~\ref{recover-classical-rem}, since $X$ is a shift space, $X$ is a subset of $\Sigma_{B_1(X)}^\textnormal{inf}$, which is the ``classical full shift" over the finite alphabet $B_1(X)$.  Hence $X$ is a shift space in the classical sense, over the finite alphabet $B_1(X)$.  This shows that any finite-symbol shift --- even one that is \emph{a priori} over an infinite alphabet --- may be viewed as a shift space over a finite alphabet.
\end{remark}

\begin{remark} \label{bases-inf-rf-fin-rem}
Recall that if $X$ is a shift space, for any $x \in \Sfin$ and finite subset $F \subseteq \A$ we define $Z_X(x) := Z(x) \cap X$ and $Z_X(x,F) := Z(x,F) \cap X$.  It follows from Theorem~\ref{basis-for-full-shift-thm} that $$ \{ Z_X(x,F) : x \in \Sfin \text{ and $F \subseteq \A$ is a finite subset} \}$$ is a basis for the topology on $X$.  If $X$ is a row-finite shift, then the set $$\{ Z_X(x) : x \in \Sfin \} \cup \{ Z_X(\0,F) : \text{ $F \subseteq \A$ is a finite subset} \}$$ is a basis for the topology on $X$.  If $X$ is a finite-symbol shift, then $$\{ Z_X(\alpha) : \alpha \in \Sfin \}$$ forms a basis for the topology on $X$.
\end{remark}

\section{Shift Morphisms and Conjugacy of Shift Spaces} \label{shift-morphisms-sec}

In the previous section we defined our basic objects of study, the shift spaces.  We now turn our attention to describing the appropriate morphisms between these objects.

\begin{definition} \label{shift-morphism-def}
Let $\A$ be an alphabet, and let $X, Y \subseteq \Sigma_\A$ be shift spaces over $\A$.  A function $\phi : X \to Y$ is a \emph{shift morphism} if the following three conditions are satisfied:
\begin{itemize}
\item[(i)] $\phi : X \to Y$ is continuous.
\item[(ii)] $\phi \circ \sigma = \sigma \circ \phi$.
\item[(iii)] $l(\phi (x)) = l(x)$ for all $x \in X$.
\end{itemize}
\end{definition}

Note that if $\phi : X \to Y$ is a shift morphism, then Condition~(iii) implies that $\phi(\Xinf) \subseteq Y^\textnormal{inf}$ and $\phi(\Xfin) \subseteq Y^\textnormal{fin}$.  When $X$ and $Y$ are infinite-symbol shift spaces, the following proposition shows Condition~(iii) may be replaced by another condition that is easier to verify.

\begin{proposition} \label{shorter-shift-verify-prop}
Let $\A$ be an alphabet, and let $X, Y \subseteq \Sigma_\A$ be infinite-symbol shift spaces over $\A$.  Then $\0 \in X$ and $\0 \in Y$, and $\phi : X \to Y$ is a shift morphism if and only if $\phi$ satisfies the following three conditions:
\begin{itemize}
\item[(i)] $\phi : X \to Y$ is continuous.
\item[(ii)] $\phi \circ \sigma = \sigma \circ \phi$.
\item[(iii')] $\phi(x) = \0$ if and only if $x = \0$.
\end{itemize}
\end{proposition}

\begin{proof}
Because $X$ and $Y$ are infinite-symbol, it follows from Proposition~\ref{fin-char-prop} that $\0 \in X$ and $\0 \in Y$.  It is clear that any shift morphism must satisfy (iii').  Suppose $\phi$ satisfies (i), (ii), and (iii').  We shall show that $\phi$ must satisfy Condition~(iii) in Definition~\ref{shift-morphism-def}.  If $x \in X$ and $l(x) = \infty$, then for every $k \in \N$ we have $\sigma^k(x) \neq \0$.  Hence for all $k \in \N$, we have $\sigma^k(\phi(x)) = \phi (\sigma^k(x)) \neq \0$, and $l(\phi(x)) = \infty$.  Likewise, if $l(x) < \infty$, then $\sigma^{l(x)} (\phi(x)) = \phi (\sigma^{l(x)} (x)) = \phi (\0) = \0$, and for any $k < l(x)$ we have $\sigma^k(x) \neq \0$ and $\sigma^k (\phi(x)) = \phi (\sigma^k(x)) \neq \0$.  Thus $l(\phi(x)) = l(x)$.
\end{proof}

\begin{remark} \label{recover-finite-morphism-rem}
If $X$ and $Y$ are finite-symbol shift spaces, then $X=\Xinf$ and $Y = Y^\textnormal{inf}$, so Condition~(iii) in Definition~\ref{shift-morphism-def} is always satisfied.  Thus if $X$ and $Y$ are finite-symbol shift spaces, a function $\phi : X \to Y$ is a shift morphism if and only if $\phi$ is continuous and $\phi \circ \sigma = \sigma \circ \phi$.  Thus we recover the ``classical definition" of a shift morphism in the finite-symbol case.  
\end{remark}

\begin{remark}
If $X$ and $Y$ are shift spaces, there are four cases to consider depending on whether each of $X$ and $Y$ is finite-symbol or infinite-symbol.   If $X$ and $Y$ are both infinite-symbol, conditions for a function $\phi : X \to Y$ to be a shift morphism are given by Proposition~\ref{shorter-shift-verify-prop}.  If $X$ and $Y$ are both finite-symbol, then Remark~\ref{recover-finite-morphism-rem} shows that a function $\phi : X \to Y$ is a shift morphism if and only if Condition~(i) and Condition~(ii) of Definition~\ref{shift-morphism-def} are satisfied.    If $X$ is finite-symbol and $Y$ is infinite-symbol, then $X = \Xinf$ and a function $\phi : X \to Y$ is a shift morphism if and only if Condition~(i) and Condition~(ii) of Definition~\ref{shift-morphism-def} are satisfied and $\phi(X) \subseteq Y^\textnormal{inf}$.  If $X$ is infinite-symbol and $Y$ is finite-symbol, then $\0 \in X$ and $\0 \notin Y$ by Proposition~\ref{fin-char-prop}, so that Condition~(iii) of Definition~\ref{shift-morphism-def} is never satisfied, and there are no shift morphisms from $X$ to $Y$.
\end{remark}

\begin{remark}
If $X$ is a shift space, the shift map $\sigma|_X : X \to X$ is not in general a shift morphism because $\sigma|_X$ is not continuous at $\0$.  In fact, $\sigma|_X : X \to X$ is a shift morphism if and only if $X$ is a finite-symbol shift space, in which case $\0 \notin X$ and $X = \Xinf$ (cf.~Proposition~\ref{fin-char-prop}).
\end{remark}

The following lemma is elementary but very useful.

\begin{lemma} \label{shift-fact-lem}
Let $\A$ be an alphabet, let $X, Y \subseteq \Sigma_\A$ be shift spaces over $\A$, and let $\phi : X \to Y$ be a shift morphism.  If $a \in \A$, $x \in \Sigma_\A$, and $ax \in X$, then $$\phi(ax) = b \phi(x)$$ for some $b \in \A$.
\end{lemma}

\begin{proof}
Since $\sigma (\phi(ax)) = \phi (\sigma (ax)) = \phi (x)$, it follows that $\phi(ax) = b\phi(x)$ for some $b \in \A$.
\end{proof}

\begin{remark}
Note that in Lemma~\ref{shift-fact-lem} although $ax \in X$ implies $x \in X$, it is not necessarily the case that $a \in X$.  Hence it does not make sense in general to apply $\phi$ to $a$.  Moreover, even when $a \in X$, it is in general not true that $\phi(ax) = \phi(a) \phi(x)$ --- so the $b$ in Lemma~\ref{shift-fact-lem} need not be $\phi(a)$ even when $\phi(a)$ is defined.
\end{remark}

\begin{definition}
Let $\A$ be an alphabet, and let $X, Y \subseteq \Sigma_\A$ be shift spaces over $\A$.  A function $\phi : X \to Y$ is a \emph{conjugacy} if $\phi$ is a shift morphism and $\phi$ is bijective. If there exists a conjugacy from $X$ to $Y$, we say $X$ and $Y$ are \emph{conjugate} and we write $X \cong Y$.
\end{definition}

\begin{remark} \label{homeo-of-conj-automatic-rem}
Since any shift space is compact, we see that if $\phi : X \to Y$ is a conjugacy, then $\phi$ must be a homeomorphism.  Thus if $\phi : X \to Y$ is a conjugacy, its set-theoretic inverse $\phi^{-1} : Y \to X$ is also a conjugacy.
\end{remark}

\begin{remark}
If we fix an alphabet $\A$, we may form a category whose objects are all shift spaces over $A$ and whose morphisms are the shift morphisms between these shift spaces.  Using Remark~\ref{homeo-of-conj-automatic-rem} we see that isomorphism in this category is precisely the relation of conjugacy.
\end{remark}

\section{Analogues of Shifts of Finite Type} \label{analogues-of-SFT-sec}

In the theory of shift spaces over finite alphabets, the most widely studied class of shift spaces are the ``shifts of finite type".  Shifts of finite type provide tractable examples of shift spaces that also have many applications throughout dynamics and other parts of mathematics.  A shift of finite type is defined to be a shift that can be described by a finite number of forbidden blocks; that is, $X = X_\F$ for a finite set $\F$ of blocks.  When the alphabet is finite, one can show that $X$ is a shift of finite type if and only if $X$ is an $M$-step shift (i.e., $X = X_\F$ for a set $\F$ with each block in $\F$ having length $M+1$) if and only if $X$ is an edge shift (i.e., $X$ is the shift space coming from a finite directed graph with no sinks where the edges are used as symbols).

In this section we shall consider analogues of the shifts of finite type for shift spaces over infinite alphabets.  Here phenomena will be more ramified, and we shall find that the shifts of finite type, the $M$-step shifts, and the edge shifts give us three distinct classes of shift spaces.

\begin{definition}
Let $\A$ be an alphabet, and let $X \subseteq \Sigma_\A$ be a shift space.  We say that $X$ is a \emph{shift of finite type} if there is a finite set of blocks $\F \subseteq \Sfin$ such that $X = X_\F$.
\end{definition}

\begin{remark}
One na\"ive approach to extending the definition of shifts of finite type to infinite alphabets is to consider ``shifts of countable type"; that is, shifts of the form $X_\F$ for a countable set of blocks $\F$.  However, if $\A$ is finite or countably infinite, then one can see that the collection $\Sfin$ of all blocks is countable, and hence when $\A$ is finite or countably infinite, every shift space will be a shift of countable type.  Thus this definition does not recover the familiar definition when the alphabet is finite, and moreover, the class of ``shifts of countable type" seems too broad to lend itself to tractable study.
\end{remark}

\begin{definition}
Let $\A$ be an alphabet, and let $X \subseteq \Sigma_\A$ be a shift space.  If $M \in \N \cup \{ 0 \}$, we say that $X$ is an \emph{$M$-step shift} if there is a set of blocks $\F \subseteq \Sfin$ such that $X = X_\F$ and $l(x) = M+1$ for all $x \in \F$.
\end{definition}

\begin{remark}
Note that a $0$-step shift is simply the full shift over the alphabet $\A \setminus \F$.  Thus our primary interest is in $M$-step shifts for $M \geq 1$.
\end{remark}

\begin{proposition} \label{same-length-prop}
Let $\A$ be an alphabet, and let $X \subseteq \Sigma_\A$ be a shift space. If there exists a set of blocks $\F \subseteq \Sfin$ and a number $M \in \N \cup \{ 0 \}$ such that $X = X_\F$ and $l(x) \leq M+1$ for all $x \in \F$, then $X$ is $M$-step.
\end{proposition}

\begin{proof}
Let $\F' := \{ uv : u \in \F, v \in \Sfin, \text{ and } l(v) = (M+1)-l(u) \}$. Then every element of $\F'$ has length $M+1$.  In addition, we see that $\Xinf_\F = \Xinf_{\F'}$, and by Corollary~\ref{Xinf-determines-X-cor} we have that $X_\F = X_{\F'}$.  Thus $X$ is $M$-step.
\end{proof}

\begin{corollary} \label{M-step-M+1-step-cor}
Any $M$-step shift is also an $(M+1)$-step shift.
\end{corollary}

\begin{corollary} \label{SFT-is-M-step-cor}
If $X$ is a shift of finite type, then $X$ is an $M$-step shift for some $M \in \N \cup \{ 0 \}$.
\end{corollary}

\begin{proof}
Let $X$ be a shift of finite type, and write $X = X_\F$ for some finite set $\F \subseteq \Sfin$.  If $\F = \emptyset$, then $X$ is a $0$-step shift.  If $\F \neq \emptyset$, then since $\F$ is finite, the value $N := \max \{ l(x) : x \in \F \}$ is a natural number.  If we define $M := N-1$, then $l(x) \leq M+1$ for all $x \in \F$, and $X$ is $M$-step by Proposition~\ref{same-length-prop}.
\end{proof}

Example~\ref{edge-not-SFT-ex} shows that the converse of this corollary is false.

\begin{definition}
A \emph{directed graph} $E=(E^0, E^1, r, s)$ consists of a countable set $E^0$ (whose elements are called vertices), a countable set $E^1$ (whose elements are called edges), a map $r:E^1 \to E^0$ identifying the range of each edge, and a map $s:E^1 \to E^0$ identifying the source of each edge.   \end{definition}

\begin{definition}
If $E=(E^0, E^1, r, s)$ is a graph, a vertex $v \in E^0$ is called a \emph{sink} if $s^{-1}(v) = \emptyset$, and $v$ is called an \emph{infinite emitter} if $s^{-1}(v)$ is an infinite set.  We write $E^0_\textnormal{sinks}$ for the set of sinks of $E$, and we write $E^0_\textnormal{inf}$ for the set of infinite emitters of $E$.  The graph $E$ is said to be \emph{finite} if both $E^0$ and $E^1$ are finite sets, and $E$ is said to be \emph{infinite} otherwise.  The graph $E$ is said to be \emph{countable} if both $E^0$ and $E^1$ are countable sets, and $E$ is said to be \emph{uncountable} otherwise.  The graph $E$ is called \emph{row-finite} if $E$ contains no infinite emitters; that is, $E$ is row-finite if and only if $E^0_\textnormal{inf} = \emptyset$.
\end{definition}

\begin{definition}
If $E=(E^0, E^1, r, s)$  is a graph, a \emph{path} in $E$ is a finite sequence of edges $\alpha := e_1 \ldots e_n$ with $r(e_i) = s(e_{i+1})$ for $1 \leq i \leq n-1$.  We say the path $\alpha$ has \emph{length} $n$, and we write $l(\alpha) = n$ to denote this.  We also let $E^n$ denote the set of paths of length $n$.  We extend the maps $r$ and $s$ to $E^n$ by defining $r(\alpha) := r(e_n)$ and $s(\alpha) := s(e_1)$.  In addition, we consider vertices to be paths of length zero, and when $\alpha = v \in E^0$, we write $l(\alpha) = 0$, $r(\alpha) := v$, and $s(\alpha) := v$.  Note that our notation of $E^0$ and $E^1$, for set of vertices and set of edges, agrees with our notation for paths of length zero and one.  We define an \emph{infinite path} in $E$ to be an infinite sequence $\alpha := e_1 e_2 \ldots$ of edges with $r(e_i) = s(e_{i+1})$ for $i \in \N$.  We write $E^\infty$ for the set of infinite paths in $E$, and we extend $s$ to $E^\infty$ by defining $s(\alpha) := s(e_1)$.
\end{definition}

\begin{definition} \label{edge-shift-closure-Einf-def}
If $E = (E^0,E^1, r, s)$ is a graph with no sinks, we let $\A := E^1$ be the alphabet consisting of the edges of $E$, and we define the \emph{edge shift} of $E$ to be the closure of $E^\infty$ inside $\Sigma_{E^1}$.
\end{definition}

\begin{proposition} \label{description-of-edge-shift-prop}
If $E$ is a graph with no sinks, then the edge shift $X_E$ is a shift space over $\A :=E^1$.  If $E^1$ is finite, then  $X_E = E^\infty$, and if $E^1$ is infinite, then
$$X_E = E^\infty \cup \{ \alpha : \text{$\alpha \in \bigcup_{n=1}^\infty E^n$ and $r(\alpha)$ is an infinite emitter} \} \cup \{ \0 \}$$ where $\0$ denotes the empty path.
In addition, in either case the length function on $\Sigma_\A$ agrees with the length function on the paths (when we take the length of the zero path $\0$ to be zero), and $\Xinf_E = E^\infty$.
\end{proposition}

\begin{proof}
When $E^1$ is finite the result is well known.  We shall prove the result when $E^1$ is infinite.  Let 
$$X := E^\infty \cup \{ \alpha : \text{$\alpha \in \bigcup_{n=1}^\infty E^n$ and $r(\alpha)$ is an infinite emitter} \} \cup \{ \0 \}.$$ 
We shall first prove that $X$ is a shift space.  To begin, we show that $X$ is closed in $\Sigma_\A$.  Suppose $\{ \alpha^n \}_{n=1}^\infty \subseteq \Sigma_\A$ is a convergent sequence of elements in $X$, and let $\alpha = \lim_{n \to \infty} \alpha^n$.  If $l(\alpha) = \infty$, then $\alpha = e_1 e_2 \ldots$ for $e_1, e_2, \ldots \in E^1$.  Since $\lim_{n \to \infty} \alpha^n = \alpha$, for any $m \in \N$ there exists $\alpha^n$ with $l(\alpha^n) \geq m+1$ and $\alpha^n = e_1 \ldots e_m e_{m+1} \beta$ for some path $\beta$.  Thus $r(e_m) = s(e_{m+1})$ for all $m \in \N$, and $\alpha = e_1 e_2 \ldots$ is an infinite path in $E$, and $\alpha \in X$.  If $l(\alpha) < \infty$, then either $\alpha = \0$ in which case $\alpha \in X$, or $\alpha = e_1 \ldots e_k$ for $e_1, \ldots, e_k \in E^1$.  In the latter case, we either have $\alpha^n$ equal to $\alpha$ for some $n$ (in which case $\alpha \in X$), or there exist infinitely many $f \in E^1$ for which some $\alpha^n$ has the form $e_1 \ldots e_k f \beta$ for some path $\beta$ in $E$.  Thus $\alpha = e_1 \ldots e_{k}$ is a finite path in $E$, and $r(\alpha)$ has infinitely many edges $f$ with source $r(e_k)$, so that $r(\alpha) = r(e_k)$ is an infinite emitter in $E$.  Hence $\alpha \in X$.

Next, we see that $\sigma(X) \subseteq X$ since the shift map $\sigma$ simply removes the first edge from any path.

Finally, we show that $X$ satisfies the infinite-extension property. Suppose that $\alpha \in X$ with $l(\alpha) < \infty$.  If $\alpha \neq \0$, then $r(\alpha)$ is an infinite emitter, and there exist infinitely many $e \in s^{-1} (r(\alpha))$.  For each such $e$ the fact that $E$ has no sinks implies that there exists an infinite path $\beta \in E^\infty$ such that $s(\beta) = r(e)$.  Thus $\alpha e \beta$ is an infinite path in $E$, and $\alpha e \beta \in X$.  If $\alpha = \0$, then for each $e \in E^1$ the fact that $E$ has no sinks implies that there exists an infinite path $\beta \in E^\infty$ such that $s(\beta) = r(e)$.  Thus $\alpha e \beta = e \beta$ is an infinite path in $E$, and $\alpha e \beta \in X$ for each $e \in E^1$.  Hence $X$ satisfies the infinite-extension property. 

Moreover, the length functions agree since they are both equal to the number of symbols (i.e. edges) appearing in a path.  It then follows that $\Xinf$ consists of the infinite paths in $X$, and $\Xinf = E^\infty$.  It follows from Proposition~\ref{Xinf-closure-is-X-prop} that $\overline{\Xinf} = X$, and thus $\overline{E^\infty} = X$.  Hence $X$ is the edge shift of $E$, and $X_E = X$.
\end{proof}

\begin{corollary}
If $E$ is a graph with no sinks, then the edge shift $X_E$ is row-finite if and only if the graph $E$ is row-finite, and the edge shift $X_E$ is finite-symbol if and only if the graph $E$ has finitely many edges. 
\end{corollary}

\begin{remark}
Note that when $E^{1}$ is infinite the elements of $X_E$ are infinite paths, finite paths of positive length ending at infinite emitters, and the zero path $\0$.  In particular, we observe that we do not include the vertices (i.e., paths of length zero) from $E$ as elements of $X_E$.  Also note that if $E^{1}$ is finite we do not include $\0$ as an element of $X_E$, while if $E^{1}$ is infinite we do include $\0$.
\end{remark}

\begin{definition}
A shift space $X$ is called an \emph{edge shift} if $X$ is conjugate to $X_E$ for a graph $E$ with no sinks.
\end{definition}

\begin{proposition} \label{edge-shifts-are-1-step-prop}
If $X_{E}$ is an edge shift, then $X_{E}$ is a $1$-step shift.
\end{proposition}

\begin{proof}
Suppose $X_{E}$ is an edge shift corresponding to a graph $E = (E^0, E^1, r, s)$.  Let $$\F := \{ ef : e,f \in E^1 \text{ and } r(e) \neq s(f) \}.$$  Then the blocks in $\F$ all have length $2$.   In addition, it is easy to verify that $\Xinf_E = \Xinf_\F = E^\infty$, and by Corollary~\ref{Xinf-determines-X-cor} we have that  $X_E = X_\F$.  Hence $X_E$ is a $1$-step shift space.
\end{proof}

The following is an example of an edge shift that is not a shift of finite type.

\begin{example} \label{edge-not-SFT-ex}
Let $E$ be the graph
\begin{equation*}
\xymatrix{ \bullet \ar[r]^{e_1} & \bullet \ar[r]^{e_2} & \bullet \ar[r]^{e_3} & \bullet
\ar[r]^{e_4} & \cdots\\ }
\end{equation*}
and let $X_E$ be the edge shift associated to $E$.    We shall argue that $X_E$ is not a shift of finite type over $\A := E^1$.  Suppose $\F$ is a finite subset of $\Sfin$.  Since $\F$ is a finite collection of finite sequences of edges, there exists $n \in \N$ such that the edge $e_n$ does not appear in any element of $\F$.  Thus the infinite sequence $e_ne_n \ldots$ is allowed, and $e_ne_n \ldots \in X_\F$.  However, $e_ne_n \ldots \notin X_E$, so $X_E \neq X_\F$.
\end{example}

The following is an example of a $1$-step shift that is not an edge shift.

\begin{example} \label{1-step-not-edge}
Let $\A = \{ a_1, a_2, a_3, \ldots \}$ be a countably infinite alphabet, and let $$\F := \{ a_ia_j : i \neq 1\text{ and } i \neq j \}.$$
Then $X_\F$ is a $1$-step shift, since every forbidden block in $\F$ has length $2$.  We shall show that $X_\F$ is not an edge shift.  For the sake of contradiction, suppose that $E$ is a graph and that $\phi : X_\F \to X_E$ is a conjugacy.  For each $i \in \N$ we see that $a_1 a_i a_ia_i \ldots $ contains no forbidden blocks, and hence $a_1 a_i a_i a_i \ldots \in X_\F$.  It follows that $a_1 = \lim_{i \to \infty} a_1 a_i a_i \ldots \in X_\F$.  Let $e \in E^1$ be the edge with $e := \phi(a_1)$.  (Note that since $\phi$ is a conjugacy it preserves length and takes elements of length one to elements of length one.)  In addition, from Lemma~\ref{shift-fact-lem} we have that for each $i \in \N$ there exists $x_i \in E^1$ with $$\phi (a_1 a_i a_i a_i \ldots) = x_i \phi(a_i a_i a_i \ldots).$$
Hence $$e = \phi(a_1) = \phi( \lim_{i \to \infty} a_1 a_i a_i a_i \ldots) =  \lim_{i \to \infty} \phi(  a_1 a_i a_i a_i \ldots) = \lim_{i \to \infty }x_i \phi(a_i a_i a_i \ldots)$$
and it follows that there exists $N \in \N$ such that $x_i = e$ for all $i \geq N$.  Since the element $a_Na_Na_N \ldots \in X_\F$ has period $1$ and $\phi$ is a conjugacy,  $\phi(a_Na_Na_N \ldots)$ has period $1$ and $\phi(a_Na_Na_N \ldots) = f f f \ldots$ for some $f \in E^1$.  (Recall that the period of an element $x$ is the smallest value of $k \in \N$ such that $\sigma^k(x) = x$.)  Thus $$\phi(a_1 a_N a_N a_N \ldots) = e \phi(a_N a_N a_N \ldots) = e fff \ldots$$
and in the graph $E$ we have $r(e) = s(f) = r(f)$.

Furthermore, since $e = \lim_{i \to \infty} e \phi(a_ia_ia_i \ldots)$, there exists $j > N$ such that $\phi(a_ja_ja_j \ldots)_1 \neq f$.  Since $a_ja_ja_j \ldots$ has period $1$, it follows $\phi(a_ja_ja_j \ldots)$ has period $1$, and thus $\phi(a_ja_ja_j \ldots) = g g g \ldots$ for some $g \in E^1$ with $g \neq f$.  Since $$\phi(a_1 a_j a_j a_j \ldots) = e \phi(a_ja_ja_j \ldots) = e g g g \ldots,$$
it follows that $r(e) = s(g) = r(g)$ in the graph $E$.

Thus there exist $f,g \in E^1$ with $s(f) =r(f) = s(g) = r(g)$ and $f \neq g$.  Hence $fgfgfg\ldots \in X_E$ and $fgfgfg\ldots$ has period $2$.  Since $\phi^{-1} : X_E \to X_\F$ is a conjugacy, $\phi^{-1}(fgfgfg\ldots)$ has period $2$, and $$\phi^{-1} (fgfgfg \ldots) = b_1b_2b_1b_2b_1b_2 \ldots \in X_\F$$ for some $b_1, b_2 \in \A$ with $b_1 \neq b_2$.  However, since $b_1 \neq b_2$, one of $b_1$ and $b_2$ is not equal to $a_1$, and hence either $b_1b_2$ is forbidden or $b_2b_1$ is forbidden.  Thus $b_1b_2b_1b_2b_1b_2 \ldots \notin X_\F$ and we have a contradiction.  Hence $X_\F$ is not an edge shift.
\end{example}

\begin{remark} \label{SFT-classes-rem}
In the ``classical situation" of shifts over finite alphabets the three classes of (1) shifts of finite type, (2) edge shifts, and (3) $M$-step shifts coincide.  When we consider shift spaces over infinite alphabets these three classes are distinct.

Proposition~\ref{edge-shifts-are-1-step-prop} and Example~\ref{1-step-not-edge} show the class of edge shifts is a proper subset of the class of $M$-step shifts.
$$ \{ \text{edge shifts} \}  \subsetneq \{ \text{$M$-step shifts} \} $$
Corollary~\ref{SFT-is-M-step-cor} and Example~\ref{edge-not-SFT-ex} show the class of shifts of finite type is a proper subset of the class of $M$-step shifts.
$$ \{ \text{shifts of finite type} \}  \subsetneq \{ \text{$M$-step shifts} \}$$
Finally, Example~\ref{edge-not-SFT-ex} gives an edge shift that is not a shift of finite type, so we know the class of shifts of finite type is not the same as the class of edge shifts.
$$ \{ \text{shifts of finite type} \}  \neq \{ \text{edge shifts} \}.$$
The authors have been unable to determine whether all shifts of finite type are edge shifts or whether there exists a shift of finite type that is not an edge shift.
\end{remark}

\noindent \textbf{Question 1:} Is every shift of finite type an edge shift?

$ $

\noindent We conjecture the answer to Question~1 is ``No".

$ $

\noindent It follows from Corollary~\ref{M-step-M+1-step-cor} that every $M$-step shift is an $(M+1)$-step shift.  These containments may be summarized as follows:
$$ \{ \text{$0$-step shifts} \}  \subseteq  \{ \text{$1$-step shifts} \}  \subseteq  \{ \text{$2$-step shifts} \} \subseteq \ldots$$
It is natural to ask if these containments are all proper.

$ $

\noindent \textbf{Question 2:} For each $M \in \N \cup \{ 0 \}$ does there exist an $(M+1)$-step shift space that is not conjugate to any  $M$-step shift?

$ $

\noindent We conjecture the answer to Question~2 is ``Yes".

\section{Row-finite Shift Spaces} \label{row-finite-analogues-of-SFT-sec}

Although the analogues of the shifts of finite type that we described in the previous section seem a bit complicated, we shall show in this section that when the shifts are also row-finite the classes of edge shifts and $M$-step shifts coincide, and we are able to obtain generalizations of some classical results.

\begin{proposition} \label{no-rf-SFTs-prop}
If $\A$ is an infinite alphabet and $X$ is a shift of finite type over $\A$, then $X$ is not row-finite.
\end{proposition}

\begin{proof}
Since $X = X_\F$ for some finite subset $\F \subseteq \Sfin$, and $\A$ is infinite, there exists an infinite sequence of distinct elements $a_1, a_2, \ldots \in \A$ such that $a_n$ does not appear in any block contained in $\F$.  If we define $x_n := a_1 a_na_na_n \ldots$, then $x_n \in X_\F$ because no subblock of $x_n$ appears in $\F$.  Thus $a_1 = \lim_{n \to \infty} x_n \in X_\F$, and $\Xfin$ contains an element other than the empty sequence $\0$.  It follows from Proposition~\ref{rf-char-prop} that $X_\F$ is not row-finite.
\end{proof}

\begin{proposition} \label{rf-1-step-is-edge-shift-prop}
If $\A$ is an alphabet and $X$ is a $1$-step shift space over $\A$ that is row-finite, then $X$ is conjugate to the edge shift of a row-finite graph.
\end{proposition}

\begin{proof}
Since $X$ is $1$-step, we may write $X = X_\F$ for some set $\F \subseteq \A^2$.  Since $X$ is row-finite, Proposition~\ref{rf-char-prop} implies that $\Xfin$ is equal to either $\{ \0 \}$ or $\emptyset$.  Define a graph $E := (E^0, E^1, r, s)$ by setting $E^0 := \A$, $E^1 := \{ (a,b) : ab \notin \F \}$, $r(a,b) = b$, and $s(a,b) = a$.  (So, in particular, we draw an edge from $a$ to $b$ if and only if $ab$ is not a forbidden block.)  Note that infinite paths in $E$ have the form $(a_1, a_2)(a_2, a_3) (a_3, a_4) \ldots$.  Because $X$ is a row-finite shift space, any symbol appearing in a sequence in $X$ can only be followed by finitely many symbols, and hence $E$ is a row-finite graph.  In addition, $E^{1}$ is infinite if and only if $X$ is infinite-symbol.

Define $\phi : X_E \to X$ as follows: 
$$\phi (\alpha) :=  \begin{cases}  
a_1 a_2 a_3 \ldots & \text{ if $\alpha = (a_1, a_2) (a_2, a_3) \ldots $} \\ 
\0 & \text{ if $\alpha = \0$} 
\end{cases}
$$
Note that $a_1 a_2 \ldots$ is allowed in $X$ if and only if $(a_1 a_2)(a_2,a_3) \ldots$ is a path in $E$.  Thus the map $\phi$ does indeed take values in $X$.  

We will show that $\phi$ is continuous, and to this end we recall from Remark~\ref{bases-inf-rf-fin-rem} that $$\{ Z(x) : x \in \Sfin \} \cup \{ Z(\emptyset, F) : \text{$F$ is a finite subset of $\A$} \}$$ is a basis for the topology on $X$.

Suppose $x = a_1 \ldots a_k$ with $k \geq 2$.  Then 
$$\phi^{-1} (Z( a_1 \ldots a_k)) = Z( (a_{1}, a_2) \ldots (a_{k-1}, a_k))$$ is an open set.  Next, suppose $x = a_1$ and let $$S = \{ b \in \A : \text{$a_1 b$ is a subblock of some element in $X$} \}.$$  Since $X$ is row-finite, $S$ is a finite set.  Also, since any nonempty sequence in $X$ is infinite, we have $Z(a_1) = \bigcup_{b \in S} Z(a_1b)$, and
$$\phi^{-1} (Z(a_1))  = \phi^{-1}\left( \bigcup_{b \in S} Z(a_1b) \right) = \bigcup_{b \in S} \phi^{-1}(Z(a_1b)) = \bigcup_{b \in S} Z(a_1,b)$$
is an open set.

Finally, suppose $x = \0$ and $F$ is a finite subset of $\A$.  Since $X$ is row-finite, for each $b \in F$ the set $S_b := \{ bc : \text{$bc$ is a subblock of some element in $X$} \}$ is finite.  Let us define $F' : \{ (b,c) : b \in F, c \in S_b \}$, which is a finite subset of the alphabet of $X_E$.  Then
\begin{align*}
\phi^{-1} (Z(\0, F)) &= \phi^{-1} \left( X \setminus \bigcup_{b \in F} Z(b) \right) = \phi^{-1} \left( X \setminus \bigcup_{b \in F, c \in S_b} Z(bc) \right) \\
&= X_E \setminus \bigcup_{b \in F, c \in S_b} \phi^{-1}(Z(bc)) = X_E \setminus  \bigcup_{b \in F, c \in S_b} Z(b,c) = Z(\0, F') 
\end{align*}
is an open set.  It follows that $\phi : X_E \to X$ is continuous.

It is straightforward to verify that $\phi \circ \sigma = \sigma \circ \phi$.  In addition, $\phi$ preserves lengths of elements. (Note that every element either has infinite length or length zero.)  Thus $\phi$ is a shift morphism.  Furthermore, it is straightforward to check that $\phi$ is bijective, so that $\phi$ is a conjugacy and $X$ is conjugate to the edge shift $X_E$.
\end{proof}

\begin{definition} \label{N-block-code-def}
Let $u=u_1 u_2 \dots u_N$ and $v=v_1 v_2 \dots v_N$ be $N$-blocks.  We say that $u$ and $v$ \emph{overlap progressively} if $u_2 u_3 \dots u_N=v_1 v_2 \dots v_{N-1}$.

Let $X$ be a row-finite shift space over the alphabet $\A$.  For any $N \in \N$, we may consider the set $B_N(X)$ of allowed $N$-blocks of $X$ and view it as an alphabet in its own right.   We define the \emph{$N$\textsuperscript{th} higher block code} $\phi_N: X \to \Sigma_{B_N(X)}$ to be the shift morphism given by
$$ (\phi_N(x))_i = x_i x_{i+1} \ldots x_{i+N-1}$$
if $x \neq \0$ and $\phi_N(\0) = \0$.
Observe that $(\phi_N(x))_i$ and $(\phi_N(x))_{i+1}$ overlap progressively.
\end{definition}

\begin{definition} \label{N-block-presentation-def}
Let $X$ be a row-finite shift space over the alphabet $\A$, let $N \in \N$, and let $\phi_N : X \to \Sigma_{B_{N} (X)}$ be the $N$\textsuperscript{th} higher block code.  Then the \emph{$N$\textsuperscript{th} higher block shift $X^{[N]}$} (or \emph{$N$\textsuperscript{th} higher block presentation of $X$}) is the image $X^{[N]} := \phi_N(X)$.
\end{definition}

\begin{proposition} \label{NSS-cong}
Let $N \in \N$ and let $X$ be a row-finite shift space over the alphabet $\A$.  Then $X^{[N]}$ is a row-finite shift space over the alphabet $B_N(X)$, and the $N$\textsuperscript{th} higher block code $\phi_N$ is a conjugacy.  Hence $X \cong X^{[N]}$.
\end{proposition}

\begin{proof}
The proof that $X^{[N]}$ is a shift space is precisely the same as in the finite case \cite[Proposition~1.4.3]{LM95}.  It is also straightforward to verify that $\phi_N$ is an injective shift morphism.  Therefore $\phi_N$ is a conjugacy, and $X \cong X^{[N]}$.
\end{proof}

\begin{proposition} \label{M-step-is-1-step-rf-prop}
If $X$ is a row-finite $M$-step shift space for $M \in \N$, then $X$ is conjugate to a row-finite $1$-step shift space.
\end{proposition}

\begin{proof}
Suppose $X$ is a row-finite $M$-step shift space over $\A$.  Since $X$ is $M$-step, we may write $X = X_\F$ for a subset $\F \subseteq \A^{M+1}$.  Since $X$ is row-finite, the $M$\textsuperscript{th} higher block shift $X^{[M]}$ is row-finite and conjugate to $X$ by Proposition~\ref{NSS-cong}.  However, $X^{[M]}$ is the shift over the alphabet $B_M(X)$ described by the forbidden blocks
\begin{align*} 
\F' &:= \{ (x_1 \ldots x_M) (x_2 \ldots x_{M+1}) : x_1 \ldots x_{M+1} \in \F \}
\\
&\qquad \cup \{ (u_{1} \ldots u_{M}) (v_{1} \ldots v_{M}) : (u_{1} \ldots u_{M}) \text{ and } (v_{1} \ldots v_{M})
\\
&\qquad \qquad \qquad \qquad  \text{do not overlap progressively} \}.
\end{align*}  
Since all blocks in $\F'$ have length $2$, the shift space $X^{[M]}$ is $1$-step.  
\end{proof}

\begin{proposition} \label{RF-classes-edge-M-step-same-prop}
If $X$ is a row-finite shift space, then the following are equivalent.
\begin{itemize}
\item[(a)] $X$ is an edge shift.
\item[(b)] $X$ is a $1$-step shift space.
\item[(c)] $X$ is an $M$-step shift space for some $M \in \N$.
\end{itemize}
\end{proposition}

\begin{proof}
$(a) \Rightarrow (c)$.  This follows from Proposition~\ref{edge-shifts-are-1-step-prop}.

$(c) \Rightarrow (b)$.  This follows from Proposition~\ref{M-step-is-1-step-rf-prop}.

$(b) \Rightarrow (a)$.  This follows from Proposition~\ref{rf-1-step-is-edge-shift-prop}.
\end{proof}

\begin{remark}
In Remark~\ref{SFT-classes-rem} we discussed the relationship among the classes of (1) shifts of finite type, (2) edge shifts, and (3) $M$-step shifts.  In the row-finite situation, these relationships take an even nicer form:  Proposition~\ref{no-rf-SFTs-prop} shows there are no row-finite shifts of finite type over infinite alphabets; i.e., 
$$ \{ \text{row-finite shifts of finite type over infinite alphabets} \} = \emptyset.$$
Proposition~\ref{RF-classes-edge-M-step-same-prop} shows that in the row-finite setting the classes of edge shifts and $M$-step shifts coincide; i.e., 
$$ \{ \text{row-finite edge shifts} \} = \{ \text{row-finite $M$-step shifts} \}.$$
Thus for row-finite shift spaces over infinite alphabets, the appropriate way to generalize shifts of finite type is not to consider shift spaces described by finite sets of forbidden blocks, but rather to study the class of edge shifts of row-finite graphs (which coincides with the class of row-finite $M$-step shifts).
\end{remark}

\section{Sliding Block Codes on Row-Finite Shift Spaces} \label{SBC-row-finite-sec}

In the theory of shifts over finite alphabets, the Curtis-Hedlund-Lyndon Theorem states that any shift morphism is equal to a sliding block code.  The way this is proven is as follows:  If $\phi : X \to Y$ is a shift morphism, then the continuity of $\phi$ and the compactness of $X$ implies that $\phi$ is uniformly continuous with respect to the standard metric on $X$ giving the topology.  Any two sequences in $X$ that are close in this metric are equal along some initial segment, and hence one may define a block map $\Phi : B_n(X) \to \A$ and use the fact that $\phi$ commutes with the shift to show $\phi$ is the sliding block code coming from $\Phi$.

For shifts over infinite alphabets, this proof does not work.  A shift morphism $\phi : X \to Y$ is continuous, and $X$ is compact, so $\phi$ is uniformly continuous with respect to any metric describing the topology on $X$.  However, two sequences that are close in such a metric need not be equal along any initial segment --- for example, their initial edges could simply be ``far out" in the space $\A_\infty$ and therefore close to $\infty$ and close to each other, but not equal.  (See the metric described in Section~\ref{metrics-sec}.)

In general, when we work over an infinite alphabet, shift morphisms will be more complicated than what occurs in the finite alphabet case.  In this section we prove that if we consider shift morphisms on row-finite shift spaces, we are able to obtain a generalized  ``sliding block code" description of the shift morphisms.  However, these sliding block codes will be of two kinds: unbounded and bounded.  The bounded sliding block codes can be written just as those in the finite alphabet case and come from a single $N$-block map, but describing the unbounded sliding block codes will require a sequence of block maps on blocks of unbounded size.

We restrict to countable alphabets in this section to obtain the strongest results that we can and also to simplify some proofs.

\begin{definition} \label{SBC-def}
If $X$ and $Y$ are shift spaces over a countable alphabet $\A$, and $X$ is row-finite, we say that a function $\phi : X \to Y$ is a \emph{sliding block code} if the following two criteria are satisfied:
\begin{itemize}
\item[(a)] If $\{ x^n \}_{n=1}^\infty \subseteq X$ and $\lim_{n \to \infty} x^n = \0$, then $\lim_{n \to \infty} \phi(x^n) = \0$.
\item[(b)] For each $a \in \A$ there exists a natural number $n(a) \in \N$ and a function $\Phi^{a} : B_{n (a)} (X) \cap Z(a) \to \A$ such that $$\phi(x_1 x_2 x_3 \ldots)_i = \Phi^{x_i}(x_i \ldots x_{n(x_i)+i-1} )$$
for all $i \in \N$ and for all $x_1 x_2 x_3 \ldots \in \Xinf$.
\end{itemize}
We say that a sliding block code is \emph{bounded} if there exists $M \in \N$ such that $n(a) \leq M$ for all $a \in \A$, and \emph{unbounded} otherwise.  We call a sliding block code $\phi$ an \emph{$M$-block code} if there exists a block map $\Phi : B_M(X) \to \A$ such that $\phi(x_1 x_2 x_3 \ldots)_i = \Phi(x_i \ldots x_{M+i-1} )$ for all $i \in \N$ and for all $x_1 x_2 x_3 \ldots \in \Xinf$.  (In other words, a sliding block code is an $M$-block code if $n(a) = M$ and $\Phi^a = \Phi |_{Z(a)}$ for all $a \in \A$.)
\end{definition}

\begin{remark}
Note that if $\0 \in X$, then Property~(a) implies that $\phi(\0) = \0$ and that $\phi$ is continuous at $\0$.  In addition, Property~(b) shows that
$$\phi(x_1 x_2 x_3 \ldots) = \Phi^{x_1}(x_1 \ldots x_{n(x_1)} ) \Phi^{x_2}(x_2 \ldots x_{n(x_2)+1} ) \Phi^{x_3}(x_3 \ldots x_{n(x_3)+2} ) \ldots$$
for all $x_1 x_2 x_3 \ldots \in \Xinf$.  This motivates the terminology: The functions $\Phi^{x_i}$ ``slide" along $x_1 x_2 x_3 \ldots$ to give the entries of $\phi(x_1 x_2 x_3 \ldots)$.
\end{remark}

\begin{lemma} \label{SBC-is-shift-morphism-rf-lem}
Let $\A$ be a countable alphabet, and let $X$ and $Y$ be shift spaces over $\A$.  If $X$ is row-finite and $\phi : X \to Y$ is a sliding block code, then $\phi$ is a shift morphism.
\end{lemma}

\begin{proof}
For each $a \in \A$ choose $n(a) \in \N$ and a function $\Phi^{a} : B_{n (a)} (X) \cap Z(a) \to \A$ such that $$\phi(x_1 x_2 x_3 \ldots) = \Phi^{x_1}(x_1 \ldots x_{n(x_1)} ) \Phi^{x_2}(x_2 \ldots x_{n(x_2)+1} ) \Phi^{x_3}(x_3 \ldots x_{n(x_3)+2} ) \ldots$$
for all $x_1 x_2 x_3 \ldots \in \Xinf$.  

We shall first show that $\phi$ is continuous.  Since $\A$ is countable,  Corollary~\ref{countability-cor} implies that $X$ is first countable and it suffices to verify that $\phi$ is sequentially continuous.  By Property~(a) of sliding block codes, $\phi$ is continuous at $\0$.  Suppose that $x \in \Xinf$ and that $\{ x^n \}_{n=1}^\infty \subseteq X$ with $\lim_{n \to \infty} x^n = x$.  Without loss of generality, we may assume that $l(x^n) = \infty$ for all $n \in \N$.  Since $\lim_{n \to \infty} x^n = x$ for any $N \in \N$ there exists a value $M \in \N$ such that $x^m_i = x_i$ for $1 \leq i \leq N + \max \{ n(x_i) : 1 \leq i \leq N \}$ whenever $m \geq M$.  Thus for all $1 \leq i \leq N$ we have
$$\phi(x)_i = \Phi^{x_i} (x_i \ldots x_{n({x_i}) + i -1}) = \Phi^{x_i} (x^m_i \ldots x^m_{n(x_i) + i -1}) = \phi(x^m)_i$$
whenever $m \geq M$.  Thus $\lim_{n \to \infty} \phi(x^n) = \phi(x)$ and $\phi$ is continuous.

Next we verify that $\phi$ commutes with the shift.  If $x \in \Xinf$, then 
\begin{align*}
\sigma (\phi(x)) &= \sigma \left( \Phi^{x_1}(x_1 \ldots x_{n(x_1)} ) \Phi^{x_2}(x_2 \ldots x_{n(x_2)+1} ) \Phi^{x_3}(x_3 \ldots x_{n(x_3)+2} \right) \ldots ) \\
&=  \Phi^{x_2}(x_2 \ldots x_{n(x_2)+1} ) \Phi^{x_3}(x_3 \ldots x_{n(x_3)+2} ) \Phi^{x_4}(x_4 \ldots x_{n(x_4)+3} ) \ldots ) \\
&=  \phi(x_2x_3x_4 \ldots) \\
&= \phi (\sigma(x))
\end{align*}
so that $\phi$ commutes with the shift map $\sigma$.

Finally, if $\0 \in X$, then $\phi(\0) = \0$ by Property~(a) in the definition of a sliding block code.  Thus $\phi$ is a shift morphism.
\end{proof}

\begin{lemma} \label{SBC-preliminary-lem}
Let $\A$ be a countable alphabet, and let $X$ and $Y$ be shift spaces over $\A$ with $X$ row-finite.  If $\phi : X \to Y$ is a shift morphism and $a \in \A$, then there exists $n \in \N$ such that whenever $u \in B_n(X) \cap Z(a)$, $ux \in X$, and $uy \in X$, then
$$\phi(ux)_1 = \phi(uy)_1.$$
\end{lemma}

\begin{proof}
Suppose the claim does not hold.  Then for each $n \in \N$ there exists $u^n \in B_n(X) \cap Z(a)$, $u^nx^n \in X$, and $u^ny^n \in X$ such that $$\phi(u^nx^n)_1 \neq \phi(u^ny^n)_1.$$
Consider the sequence $\{ u^nx^n \}_{n=1}^\infty$ and $\{ u^n y^n \}_{n=1}^\infty$ in $Z(a)$.  Since $Z(a)$ is compact by Lemma~\ref{cylinder-compopen-lem}, there exists convergent subsequences $\{ u^{n_k}x^{n_k} \}_{k=1}^\infty$ and $\{ u^{n_k}y^{n_k} \}_{k=1}^\infty$ converging to a (nonzero) element of $Z(a)$. Since $X$ is row-finite, this implies that there exists $z \in X$ with $l(z) = \infty$ and $\lim_{k \to \infty} u^{n_k} x^{n_k} = z$.  Because $n_1 < n_2 < n_3 < \ldots$ and $l(u^{n_k}) = n_k$, we see that $\lim_{k \to \infty} u^{n_k} y^{n_k} = z$ also.  However, since $$\lim_{k \to \infty} u^{n_k} x^{n_k} = \lim_{k \to \infty} u^{n_k} y^{n_k} = z,$$ the continuity of $\phi$ implies that $$\lim_{k \to \infty} \phi(u^{n_k} x^{n_k}) = \lim_{k \to \infty} \phi(u^{n_k} y^{n_k}) = \phi(z).$$  Since $l(z) = \infty$, we know $l(\phi(z)) = \infty$.  Hence for a large enough $k$ we have $$\phi(u^{n_k}x^{n_k})_1 = \phi(u^{n_k}y^{n_k})_1,$$ which is a contradiction.
\end{proof}

\begin{proposition} \label{shift-mophism-is-SBC-rf-prop}
Let $\A$ be a countable alphabet, and let $X$ and $Y$ be shift spaces over $\A$.  If $X$ is row-finite and $\phi : X \to Y$ is a shift morphism, then $\phi$ is a sliding block code.  Moreover, if $\phi$ is a bounded sliding block code with bound $M$, then we may choose $\Phi : B_M (X) \to \A$ such that 
$$\phi(x_1 x_2 x_3 \ldots) = \Phi(x_1 \ldots x_M) \Phi(x_2 \ldots x_{M+1}) \Phi (x_3 \ldots x_{M+2}) \ldots$$
and $\phi$ is an $M$-block code.
\end{proposition}

\begin{proof}
Since $\phi$ is a shift morphism, if $\0 \in X$ then $\phi(\0) = \0$ and $\phi$ is continuous at $\0$.  Thus Property~(a) of Definition~\ref{SBC-def} is satisfied.

We now establish Property~(b) of Definition~\ref{SBC-def}.  For each $a \in \A$, Lemma~\ref{SBC-preliminary-lem} implies that there exists $n(a) \in \N$ such that  whenever $u \in B_{n(a)}(X) \cap Z(a)$, $ux \in X$, and $uy \in X$, then
\begin{equation} \label{well-def-SBC-eq}
\phi(ux)_1 = \phi(uy)_1.
\end{equation}
Define the map $\Phi^a : B_{n(a)}(X) \cap Z(a) \to \A$ as follows: For each $u \in B_{n(a)}(X) \cap Z(a)$, the fact that $X$ is row-finite implies that $u$ is a subblock of an infinite sequence in $X$, and by shifting we may choose $u$ as the initial segment of an infinite path.  Thus there exists $x^n \in X$ such that $ux^n \in X$.  We then define
$$\Phi^a (u) := \phi(ux^n)_1.$$
We see that $\phi$ is well defined by \eqref{well-def-SBC-eq}.  Furthermore, for all $x \in \Xinf$ we have
$$\phi(x)_i = \sigma^{i-1}(\phi(x))_1 = \phi (\sigma^{i-1}(x))_1 = \phi( x_i x_{i+1} \ldots)_1 = \Phi^{x_i} (x_i \ldots x_{n(x_i) + i -1}).$$
Thus $\phi$ is a sliding block code.

Moreover, if there exists $M \in \N$ such that $n(a) \leq M$ for all $a \in \A$, then when defining $\Phi^a$ above we can instead define $\Phi^a : B_M(X) \to \A$ by $\Phi^a(u) := \phi(ux)_1$ for any $x \in X$ with $ux \in X$.  This function is likewise well defined, and so the $\Phi^a$ agree for all $a \in \A$.  Thus we let $\Phi := \Phi^a$ for any (and hence all) $a \in \A$, and $\phi$ is an $M$-block code.
\end{proof}

The following may be thought of as a generalization of the Curtis-Hedlund-Lyndon Theorem to row-finite shift spaces over countable alphabets.

\begin{theorem} \label{SBC-ff-shift-morphism-thm}
Let $\A$ be a countable alphabet, and let $X$ and $Y$ be shift spaces over $\A$.  If $X$ is row-finite and $\phi : X \to Y$ is a function, then $\phi$ is a shift morphism if and only if $\phi$ is a sliding block code.  Moreover, if $\phi$ is a bounded sliding block code, then $\phi$ is an $M$-block code from some $M \in \N$.
\end{theorem}

\begin{proof}
The result follows from Lemma~\ref{SBC-is-shift-morphism-rf-lem} and Proposition~\ref{shift-mophism-is-SBC-rf-prop}.
\end{proof}

\begin{proposition} \label{bounded-composed-is-bounded-prop}
Let $X$, $Y$, and $Z$ be shift spaces over a countable alphabet $\A$, and let $X$ and $Y$ be row-finite.  If $\phi : X \to Y$ is a bounded sliding block code, and $\psi : Y \to Z$ is a bounded sliding block code, then the composition $\psi \circ \phi : X \to Z$ is a bounded sliding block code.  Moreover, if $\phi$ is an $M$-block code and $\psi$ is an $N$-block code, then $\psi \circ \phi$ is an $(M+N-1)$-block code.
\end{proposition}

\begin{proof}
Since $\phi$ and $\psi$ are bounded sliding block codes, it follows from Theorem~\ref{SBC-ff-shift-morphism-thm} that $\phi$ is an $M$-block code and $\psi$ is an $N$-block code for some $M, N \in \N$.  Let $\Phi : B_M(X) \to \A$ be an $M$-block map for $\phi$ and let $\Psi : B_{N} (Y) \to \A$ be an $N$-block map for $\psi$, so that $\phi(x_1 x_2 x_3 \ldots)_i = \Phi(x_i \ldots x_{M+i-1} )$ and $\psi(y_1 y_2 y_3 \ldots)_i = \Psi(y_i \ldots y_{N+i-1} )$ for all $i \in \N$ and for all $x_1 x_2 x_3 \ldots \in \Xinf$ and $y_1 y_2 y_3 \ldots \in Y^{\text{inf}}$. 

Since $\phi$ and $\psi$ are each continuous at $\0$, it follows that $\psi \circ \phi$ is continuous at $\0$ and Property~(a) from Definition~\ref{SBC-def} is satisfied.  To verify Property~(b), we define an $(M+N-1)$-block map $\Delta : B_{M+N-1} (X) \to \A$ by $$\Delta (a_1 \ldots a_{M+N-1}) := \Psi ( \Phi (a_1 \ldots a_M) \Phi (a_2 \ldots a_{M+1}) \ldots \Phi (a_N \ldots a_{M+N-1})).$$  Then for any $x_1 x_2 x_3 \ldots \in \Xinf$ and any $i \in \N$ we have
\begin{align*}
\psi \circ \phi (x_1 &x_2  x_3  \ldots)_i \\
&= \Psi ( \phi (x_1 x_2 x_3 \ldots)_i \phi (x_1 x_2 x_3 \ldots)_{i+1} \ldots \phi (x_1 x_2 x_3 \ldots)_{i+N-1})  \\
&= \Psi ( \Phi(x_i \ldots x_{M+i-1}) \Phi(x_{i+1} \ldots x_{M+i}) \ldots \Phi(x_{N+i-1} \ldots x_{M+N+i-2})) \\
&= \Delta (x_i \ldots x_{M+N+i-2})
\end{align*}
so that $\psi \circ \phi$ is an $M+N-1$-block map, and $\psi \circ \phi$ is a bounded sliding block code.
\end{proof}

\begin{proposition} \label{higher-block-code-is-conj-prop}
Let $\A$ be a countable alphabet, and let $X$ be a row-finite shift space over $\A$.  Choose $M \in \N$, let $X^{[M]}$ be the $M$\textsuperscript{th} higher block presentation of $X$, and let  $\phi_M : X \to X^{[M]}$ be the $M$\textsuperscript{th}-higher block code from $X$ to $X^{[M]}$ (see Definition~\ref{N-block-code-def} and Definition~\ref{N-block-presentation-def}).  Then $\phi_M :  X \to X^{[M]}$ is an $M$-block code that is also a conjugacy, and moreover, the inverse of $\phi_M$ is a $1$-block code.
\end{proposition}

\begin{proof}
It is clear that the shift morphism $\phi_M : X \to X^{[M]}$ is an $M$-block code coming from the $M$-block map $\Phi : B_M(X) \to B_1(X^{[M]})$ given by $\Phi (a_1 \ldots a_M) = a_1 \ldots a_M$.   (Recall that $B_1(X^{[M]}) = B_M(X)$.)  Let $\pi_M : X^{[M]} \to X$ be the shift morphism coming from the $1$-block map $\Pi : B_1(X^{[M]}) \to \A$ given by $\Pi (a_1 \ldots a_m) = a_1$.  One can easily verify that $\pi$ is an inverse for $\phi$. 
\end{proof}

\begin{proposition}[Recoding to a $1$-block code] \label{recodeone}
Let $\A$ be a countable alphabet, let $X$ and $Y$ be shift spaces over $\A$, and suppose that $X$ is row-finite.  If $\psi : X \to Y$ is an $M$-block code, and if $X^{[M]}$ denotes  the $M$\textsuperscript{th} higher block presentation of $X$ and $\phi_M : X \to X^{[M]}$ denotes the $M$\textsuperscript{th}-higher block code from $X$ to $X^{[M]}$ (see Definition~\ref{N-block-code-def} and Definition~\ref{N-block-presentation-def}), then there exists a $1$-block code $\psi^{[M]}: X^{[M]} \to Y$ such that $\psi^{[M]} \circ \phi_M= \psi$.  In particular, the following diagram commutes:
\[\begin{tikzpicture}
   \node[inner sep=0pt, circle] (a) at (-1, 1) {$X$}; 
   \node[inner sep=0pt, circle] (b) at (-1, -0.5) {$Y$}; 
   \node[inner sep=0pt, circle] (d) at (1, 1) {$X^{[M]}$}; 
     \draw[style=semithick, -latex] (a.south)--(b.north) node[pos=0.5, anchor=east, inner sep=1pt]{$ \psi$};
    \draw[style=semithick, -latex] (a.east)--(d.west) node[pos=0.5, anchor=south, inner sep=3pt]{$\phi_M$} node[pos=0.4, anchor=north, inner sep=3pt]{$\cong$};
      \draw[style=semithick, -latex] (d.south west)--(b.east) node[pos=0.5, anchor=west, inner sep=7pt]{$\psi^{[M]}$};
\end{tikzpicture}\]
\end{proposition}

\begin{proof}
Let $\Psi : B_M(X) \to \A$ be an $M$-block map that defines $\psi$.  Define the $1$-block code $\psi^{[M]}: X^{[M]} \to Y$ by $\psi^{[M]} (x)_i := \Psi (x_i)$. Then for all $i \in \N$ we have
\begin{align*}
(\psi^{[M]} \circ \phi_M (x))_i &=\Psi ( \phi_M(x)_i) = \Psi (x_i \ldots x_{i+M-1}) = \psi(x)_i,
\end{align*}
and $\psi^{[M]} \circ \phi_M= \psi$.
\end{proof}

\begin{remark} \label{rem-recodeone}
Note that if the map $\psi$ in Proposition~\ref{recodeone} is a conjugacy, then the fact that $\phi_M$ is a conjugacy (see Proposition~\ref{higher-block-code-is-conj-prop}) implies that $\psi^{[M]}$ is also a conjugacy.  Moreover, if $\psi$ is a conjugacy with bounded inverse, then Proposition~\ref{bounded-composed-is-bounded-prop}  and Proposition~\ref{higher-block-code-is-conj-prop} imply that $\psi^{[M]}$ has a bounded inverse.
\end{remark}

We now describe a method to determine if a sliding bock code is bounded.

\begin{definition} \label{boundedness-metric-def}
Let $X$ be a shift space.  We define the \emph{boundedness metric} $D$ on $\Xinf$ by
$$D(x ,y) := \begin{cases} 1/2^n & \text{if $n$ is the smallest value such that $x_n \neq y_n$}  \\
0 & \text{if $x=y$} \\ 
\end{cases} $$
for all $x,y \in \Xinf$.
\end{definition}

\begin{remark}
Note that if $X$ is infinite-symbol, the boundedness metric $D$ on $\Xinf$ cannot be extended to a metric on $X$ that gives the topology of $X$.  In particular, if $\{ x^n \}_{n=1}^\infty \subseteq X$ and $\lim_{n \to \infty} x^n = \0$, then the sequence $\{ x^n \}_{n=1}^\infty$ is not Cauchy with respect to $D$ --- indeed there are arbitrarily large $m,n \in \N$ with $D(x^m, x^n) = 1/2$.

On the other hand, if $X$ is finite-symbol, then $D$ induces the usual metric on $\Xinf = X$.
\end{remark}

\begin{proposition} \label{bounded-fff-uniformly-cts-prop}
Let $\A$ be a countable alphabet, and let $X$ and $Y$ be shift spaces over $\A$.  If $X$ is row-finite and $\phi : X \to Y$ is a sliding block code, then $\phi$ is bounded if and only if $\phi|_{\Xinf}$ is uniformly continuous with respect to the boundedness metric $D$ of Definition~\ref{boundedness-metric-def}.
\end{proposition}

\begin{proof}
Suppose $\phi$ is a bounded sliding block code, and let $M \in \N$ and $\Phi : B_M(X) \to \A$ such that $\phi(x)_i = \Phi(x_i \ldots x_{i+M-1})$ for all $i \in \N$.  Let $\epsilon > 0$.  Choose $N \in \N$ such that $1/2^N < \epsilon$.  Let $\delta := 1/2^{M+N}$.  If $x,y \in \Xinf$ and $D(x,y) < \delta$, then $D(x,y) < 1/2^{M+N}$, and $x_i = y_i$ for all $1 \leq i \leq M+N$.  It follows that
$$\phi(x)_i = \Phi(x_i \ldots x_{i+M-1}) = \Phi(y_i \ldots y_{i+M-1}) = \phi(y)_i$$
for all $1 \leq i \leq N$.  Thus $D(\phi(x), \phi(y)) < 1/2^N < \epsilon$, and $\phi|_{\Xinf}$ is uniformly continuous with respect to the boundedness metric $D$.

Conversely, suppose that $\phi$ is uniformly continuous with respect to the boundedness metric $D$.  Then there exists $\delta > 0$ such that if $x, y \in \Xinf$ and $D(x,y) < \delta$, then $d(\phi(x), \phi(y)) < 1/2$.  Choose $M \in \N$ such that $1/2^M < \delta$.  Define a block map $\Phi : B_M(X) \to \A$ as follows: For any $u \in B_{M} (X)$, $u$ is a subblock of an element in $\Xinf$.  Since $\sigma(X) \subseteq X$ we may choose $u$ so that  it is the initial segment of this element, and thus there exists $x \in X$ with $ux \in X$.  We then define $\Phi(u) := \phi(ux)_1$.  To see that $\Phi$ is well defined, suppose that $x,y \in X$ with $ux, uy \in X$.  Then $D(ux,uy) < 1/2^M < \delta$, and hence $D(\phi(ux), \phi(uy)) < 1/2$.  It follows that $\phi(ux)_1 = \phi(uy)_1$, and $\Phi$ is well-defined.  We shall now show that $\phi$ is equal to the sliding block code determined by $\Phi$.  Suppose $x \in X$.  Then the fact that $\phi$ commutes with the shift map implies that 
$$\phi(x)_i = \sigma^{i-1} ( \phi(x))_1 = \phi(\sigma^{i-1}(x))_1 = \phi (x_i x_{i+1} \ldots)_1 = \Phi(x_i \ldots x_{i+M-1}).$$ Thus $\phi$ is the $M$-block code determined by $\Phi$, and $\phi$ is bounded.
\end{proof}

\begin{corollary}
If $X$ is a finite-symbol shift space, then any shift morphism is an $M$-block code for some $M \in \N$.
\end{corollary}

\begin{proof}
Since $X$ is finite-symbol, the topology induced by the boundedness metric $D$ is equal to the topology on $\Xinf = X$.  Since $X$ is a shift space, $X$ is compact.  Thus any shift morphism, which is continuous by definition, must be uniformly continuous with respect to the metric $D$.    It follows from Theorem~\ref{SBC-ff-shift-morphism-thm} that any shift morphism is a sliding block code, and it follows from Proposition~\ref{bounded-fff-uniformly-cts-prop} that any sliding block code is bounded, and hence an $M$-block code for some $M \in \N$.
\end{proof}

\section{Symbolic Dynamics and $C^*$-algebras} \label{C*-algebras-sec}

Shifts of finite type have a long history of interaction with $C^*$-algebras of graphs.  Since the seminal work of Cuntz and Krieger in the early 1980's, it has been known that if two finite graphs have conjugate edge shifts, then the Cuntz-Krieger algebras (i.e., the graph $C^*$-algebras) of those graphs are isomorphic.

In this section we show how our formulation of shift spaces over infinite alphabets allows us to extend this result to infinite graphs.  In particular, we show that if $E$ and $F$ are countable graphs, and the edge shifts $X_E$ and $X_F$ are conjugate in the sense we have described in this paper, then the graph $C^*$-algebras $C^*(E)$ and $C^*(F)$ are isomorphic.  This gives credibility to our definition of shift spaces over countable alphabets, and shows that --- at the very least --- it is a viable definition for graph $C^*$-algebras.

\subsection{$C^*$-algebras of Countable Graphs and their Groupoids} \label{C*-alg-groupoid-subsec}

\begin{definition} \label{graph-C*-def}
If $E = (E^0, E^1, r, s)$ is a directed graph, the \emph{graph $C^*$-algebra} $C^*(E)$ is the universal
$C^*$-algebra generated by mutually orthogonal projections $\{ p_v
: v \in E^0 \}$ and partial isometries with mutually orthogonal
ranges $\{ s_e : e \in E^1 \}$ satisfying
\begin{enumerate}
\item $s_e^* s_e = p_{r(e)}$ \quad  for all $e \in E^1$
\item $s_es_e^* \leq p_{s(e)}$ \quad for all $e \in E^1$
\item $p_v = \sum_{\{ e \in E^1 : s(e) = v \}} s_es_e^* $ \quad for all $v \in E^0$ with $0 < |s^{-1}(v)| < \infty$.
\end{enumerate}
\end{definition}
\begin{definition} We call Conditions (1)--(3) in Definition~\ref{graph-C*-def} the \emph{Cuntz-Krieger relations}.  Any collection $\{ S_e, P_v : e \in E^1, v \in E^0 \}$ where the $P_v$ are mutually orthogonal projections, the $S_e$ are partial isometries with mutually orthogonal ranges, and the Cuntz-Krieger relations are satisfied is called a \emph{Cuntz-Krieger $E$-family}.  For a path $\alpha := e_1 \ldots e_n$, we define $S_\alpha := S_{e_1} \ldots S_{e_n}$ and when  $\alpha = v$ is a vertex we define $S_\alpha := P_v$.
\end{definition}

\begin{definition}
If $E = (E^0, E^1, r, s)$ is a directed graph with no sinks, the \emph{boundary path space} of $E$ is the set
$$\Pa E := E^\infty \cup \{ \alpha \in \bigcup_{n=0}^\infty E^n : r(\alpha) \text{ is an infinite emitter} \}.$$
For each $\alpha \in \bigcup_{n=0}^\infty E^n$ and any finite subset $F \subseteq E^1$, we define
$$Z(\alpha, F) := \{ \alpha \gamma: \gamma \in \Pa E, r(\alpha) = s(\gamma), \text{ and } \gamma_1 \notin F \}.$$
We give $\Pa E$ the topology whose basis consists of the sets
$$\{ Z(\alpha, F) : \alpha \in \bigcup_{n=0}^\infty E^n \text{ and $F \subseteq E^1$ is finite} \}.$$
\end{definition}

\begin{remark} \label{path-space-edge-shift-comparison}
Note that the boundary path space $\Pa E$ includes some of the vertices --- namely the infinite emitters --- and that in the basis elements $Z(\alpha, F)$, the path $\alpha$ is allowed to be an infinite emitter.  The boundary path space is similar to the edge shift of a graph, but not exactly the same.  In fact, if $E$ is a graph, $\Pa E$ is the boundary path space, $E^0_{\textnormal{inf}}$ denotes the set of infinite emitters in $E$, and $X_E$ is the edge shift of $E$, then
$$\Pa E = (X_E \setminus \{ \0 \}) \cup E^0_{\textnormal{inf}}$$
as sets.  Moreover, with the topologies defined on $\Pa E$ and $X_E$, one can see that if $\{ x^n \}_{n=1}^\infty \subseteq X_E \setminus \{ \0 \}$, and $\lim_{n \to \infty} x^n = x$ with $x \in X_E \setminus \{ \0 \}$, then $\lim_{n \to \infty} x^n = x$ in $\Pa E$.  Similarly, if  $\{ x^n \}_{n=1}^\infty \subseteq \Pa E$, $\lim_{n \to \infty} x^n = x$ in $\Pa E$, $l(x) \geq 1$, and $\l(x^n) \geq 1$ for all $n \in \N$, then $\lim_{n \to \infty} x^n = x$ in $X_E$.

Also note that if $\{ e_1, e_2, \ldots \}$ are distinct edges in $E$, then $\lim_{n \to \infty} e_n = \0$ in $X_E$, whereas the sequence $\{ e_n \}_{n=1}^\infty$ converges in $\Pa E$ if and only if there is a vertex $v \in E^0_{\textnormal{inf}}$ and $N \in \N$ such that $s(e_n) = v$ for all $n \geq N$, and in this case $\lim_{n \to \infty} e_n = v$ in $\Pa E$.
\end{remark}

\begin{definition}
If $E = (E^0, E^1, r, s)$ is a directed graph with no sinks, the \emph{graph groupoid} $\G_E$ is defined by
\begin{align*}
\G_E := \{ (\alpha \gamma, l(\alpha) - l(\beta), \beta \gamma) &: \alpha, \beta \in  \bigcup_{n=0}^\infty E^n,
\\
&\qquad \gamma \in \Pa E, \text{ and } r(\alpha) = r(\beta) = s(\gamma) \}.
\end{align*}
If $(x,k,y) \in \G_E$ we define the range map $r(x,k,y) = y$ and source map $s(x,k,y) = x$.  Two elements $(x,k,y), (y', l, z) \in \G_E$ are composable if and only if $y=y'$, in which case we define $(x,k,y) \cdot (y,l,z) := (x, k+l, z)$.  The inverse of $(x,k,y) \in \G_E$ is defined to be $(x,k,y)^{-1} = (y, -k, x)$.  If $\alpha, \beta \in  \bigcup_{n=0}^\infty E^n$ with $r(\alpha) = r(\beta) = v$, and if $F$ is a finite subset of $E^1$, we define
$$Z(\alpha, \beta, F) := (\alpha \gamma, l(\alpha)-l(\beta), \beta \gamma) : \gamma \in \Pa E, \ s(\gamma) = v, \text{ and } \gamma_1 \notin F \}.$$  We give $\G_E$ the topology generated by the collection of all $Z(\alpha, \beta, F)$.
\end{definition}

\begin{remark}
Paterson shows in \cite{Pat} that the collection $Z(\alpha, \beta, F)$ forms a basis for a locally compact Hausdorff topology making $\G_E$ into a topological groupoid, that the source map $(x,k,y) \mapsto x$ is a local homeomorphism from $\G_E$ to the boundary path space $\Pa E$, and that the groupoid $C^*$-algebra $C^*(\G_E)$ is isomorphic to the graph $C^*$-algebra $C^*(E)$.
\end{remark}

\begin{proposition} \label{path-homeo-from-conj-prop}
Let $E$ and $F$ be countable graphs with no sinks and no sources, and suppose that $\phi : X_E \to X_F$ is a conjugacy.  The function $\widetilde{\phi} : \Pa E \to \Pa F$ defined by
$$ \widetilde{\phi} (\alpha) := \begin{cases} \phi(\alpha) & \text{ if $l(\alpha) \geq 1$} \\ r(\phi(e)) & \text{ if $\alpha = v \in E^0_{\textnormal{inf}}$ and $e \in r^{-1}(v)$} \end{cases} $$ is well defined and a homeomorphism from $\Pa E$ onto $\Pa F$.
\end{proposition}

\begin{proof}
We first show that $\widetilde{\phi}$ is well defined.  The only issue is when $\alpha = v \in E^0_{\textnormal{inf}}$ is an infinite emitter.  Choose $e,f \in E^1$ with $r(e) = r(f) = v$.  Choose an infinite sequence of distinct edges $e_1, e_2, \ldots \in s^{-1}(v)$, and for each $n \in \N$, choose $x^n \in E^\infty$ such that $e_nx^n \in E^\infty$.  By Lemma~\ref{shift-fact-lem} for each $n \in \N$ there exists $b_n \in E^1$ such that
$$\phi (ee_nx^n) = b_n \phi(e_n x^n)$$
and there exists $c_n \in E^1$ such that
$$\phi(fe_nx^n) = c_n \phi(e_nx^n).$$
Since $$\phi(e) = \phi ( \lim_{n \to \infty} ee_nx^n) = \lim_{n \to \infty} \phi(ee_nx^n) = \lim_{n \to \infty} b_n \phi(e_nx^n)$$ there exists $N \in \N$ such that $n \geq N$ implies $b_n = \phi(e)$.  Likewise, since $$\phi(f) = \phi ( \lim_{n \to \infty} fe_nx^n) = \lim_{n \to \infty} \phi(fe_nx^n) = \lim_{n \to \infty} c_n \phi(e_nx^n)$$ there exists $M \in \N$ such that $n \geq M$ implies $c_n = \phi(f)$.  If we let $K := \max \{ N, M \}$, then $$ \phi (e e_K x^K)=\phi(e) \phi(e_K x^K)$$ so that $r(\phi(e)) = s(\phi(e_K x^K))$, and $$\phi (f e_K x^K)=\phi(f) \phi(e_K x^K)$$ so that $r(\phi(f)) = s(\phi(e_K x^K))$.  Hence $r(\phi(e)) = r(\phi(f))$ and $\widetilde{\phi}$ is well-defined.

Next, we show that $\widetilde{\phi}$ is continuous.  Since $E$ is countable, $\Pa E$ is first countable, and it suffices to show that $\widetilde{\phi}$ is sequentially continuous.  Throughout we shall use the fact that, since  $\phi$ is a conjugacy, $\phi$ preserves lengths.  Let $\{ x^n \}_{n=1}^\infty \subseteq \Pa E$ be a convergent sequence with $x = \lim_{n \to \infty} x^n$ in $\Pa E$.  

If $l(x) \geq 1$, then eventually $l(x^n) \geq 1$ and $x^n \in X_E$.  (See Remark~\ref{path-space-edge-shift-comparison}.)  Thus $x = \lim_{n \to \infty} x^n$ in $X_E$, and hence $\lim_{n \to \infty} \widetilde{\phi} (x^n) = \lim_{n \to \infty} \phi(x^n) = \phi (\lim_{n \to \infty} x^n) = \phi(x) = \widetilde{\phi} (x)$ in $X_F$ and also $\Pa F$.    If $l(x) = 0$, then $x = v \in E^0_{\textnormal{inf}}$ is an infinite emitter, and eventually $s(x^n) = v$.  Since $E$ has no sources, we may choose an edge $e \in E^1$ with $r(e) = v$.  Eventually $ex^n$ is a path in $\Pa E$, and $\lim_{n \to \infty} ex^n = e$ in $\Pa E$, so that $\lim_{n \to \infty} ex^n = e$ in $X_E$.  (See Remark~\ref{path-space-edge-shift-comparison}.)  As before, this implies that $\lim_{n \to \infty} \widetilde{\phi} (ex^n) = \widetilde{\phi}(e)$ in $\Pa F$.  If $l(x^n) \geq 1$, then by Lemma~\ref{shift-fact-lem} we have that $\phi(ex^n) = b_n \phi(x^n) = b_n \widetilde{\phi}(x^n)$ 
for some $b_n \in F^1$.  Since $\lim_{n \to \infty} \phi(ex^n) = \phi(e)$ in $X_F$ there exists $N \in \N$ such that $$n \geq N \text{ and } l(x^n) \geq 1 \implies \phi(ex^n) = \phi(e) \phi(x^n)$$ and hence $n \geq N$ and $l(x^n) \geq 1$ implies $s(\phi(x^n)) = r(\phi(e)) = \widetilde{\phi} (v).$ 
If $l(x^n) = 0$, then $x^n$ is an infinite emitter and there exists $M \in \N$ such that 
$$n \geq M \text{ and } l(x^n) = 0 \implies x^n = r(e) = v.$$  Hence for $n \geq M$ and $l(x^{n}) = 0$ we have
$$\phi(ex^n) = \phi(e) = \phi(e) r(\phi(e)) = \phi(e) \widetilde{\phi}(r(e)) = \phi(e) \widetilde{\phi} (x^n),$$ and thus  $\widetilde{\phi}(x^n) = r(\phi(e)) = \widetilde{\phi}(v)$.
If we let $K := \max \{ N, M \}$, then for any $n \geq K$ we have 
$$ \widetilde{\phi}(ex^n) = \phi(ex^n) = \phi(e) \widetilde{\phi} (x^n)$$
and therefore
$$\lim_{n \to \infty} \widetilde{\phi}(ex^n) = \lim_{n \to \infty} \phi(e) \widetilde{\phi} (x^n)$$
in $\Pa F$.  Since $\lim_{n \to \infty} \widetilde{\phi}(ex^n) = \widetilde{\phi} (e) = \phi(e)$ in $\Pa F$, we have $\lim_{n \to \infty} \widetilde{\phi} (x^n) = r(\phi(e)) = \widetilde{\phi} (v)$.  Thus $\widetilde{\phi}$ is continuous.

We have shown that $\widetilde{\phi} : \Pa E \to \Pa F$ is well defined and continuous.  To show $\widetilde{\phi}$ is a homeomorphism, we simply define $\widetilde{\phi}^{-1} : \Pa F \to \Pa E$ by
$$ \widetilde{\phi}^{-1} (\alpha) := \begin{cases} \phi^{-1}(\alpha) & \text{ if $l(\alpha) \geq 1$} \\ r(\phi^{-1}(f)) & \text{ if $\alpha = v \in F^0_{\textnormal{inf}}$ and $f \in r^{-1}(v)$.} \end{cases}$$ By a symmetric argument, $\widetilde{\phi}^{-1}$ is well defined and continuous, and it is straightforward to check that $\widetilde{\phi}^{-1}$ is the inverse of $\widetilde{\phi}$.  Thus $\widetilde{\phi}$ is a homeomorphism.
\end{proof}

\begin{theorem} \label{shifts-conjugate-implies-groupoids-iso-thm}
Let $E$ and $F$ be countable graphs with no sinks and no sources.  If $X_E \cong X_F$, then $\G_E \cong \G_F$.
\end{theorem}

\begin{proof}
Let $\phi : X_E \to X_F$ be a conjugacy, and define $H : \G_E \to \G_F$ by $$H(\alpha \gamma, l(\alpha) - l(\beta), \beta \gamma) := (\widetilde{\phi} (\alpha \gamma), l(\alpha) - l(\beta), \widetilde{\phi} (\beta \gamma))$$ where $\widetilde{\phi} : \Pa E \to \Pa F$ is the homeomorphism defined in Proposition~\ref{path-homeo-from-conj-prop}.

Observe that if either $\alpha \gamma$ or $\beta \gamma$ is an infinite emitter, then $(\widetilde{\phi} (\alpha \gamma), l(\alpha) - l(\beta), \widetilde{\phi} (\beta \gamma)) \in \G_F$ because $\widetilde{\phi}$ preserves lengths.  If $\gamma$ has positive length, then 
\begin{align*}
\sigma^{l(\alpha)} (\widetilde{\phi} (\alpha \gamma) ) &= \sigma^{l(\alpha)}( \phi (\alpha \gamma) )=  \phi (\sigma^{l(\alpha)} (\alpha \gamma) ) = \phi (\gamma) = \phi (\sigma^{l(\beta)} (\beta \gamma) )  = \\
&= \sigma^{l(\beta)}( \phi (\beta \gamma) ) = \sigma^{l(\beta)} (\widetilde{\phi} (\beta \gamma) ). 
\end{align*}
Hence $\widetilde{\phi} (\alpha \gamma) = \delta  \phi (\gamma)$ for some path $\delta$ in $F$ with $l(\delta) = l(\alpha)$ and $\widetilde{\phi} (\beta \gamma) = \epsilon  \phi (\gamma)$ for some path $\epsilon$ in $F$ with $l(\epsilon) = l(\beta)$, and hence $(\widetilde{\phi} (\alpha \gamma), l(\alpha) - l(\beta), \widetilde{\phi} (\beta \gamma)) \in \G_F$.  Thus $H$ does indeed take values in $\G_F$.

To see that $H$ is continuous, we observe that the following diagram commutes:
$$
\xymatrix{
\G_E \ar[r]^H \ar[d]_{r_{\G_E}} & \G_F \ar[d]^{r_{\G_F}} \\
\Pa E \ar[r]^{\widetilde{\phi}} & \Pa F
}
$$
$$ $$
where $r_{\G_E}$ and $r_{\G_F}$ denote the range maps on $\G_E$ and $\G_F$, respectively.  Since $r_{\G_E}$ and $r_{\G_F}$ are local homeomorphisms, and $\widetilde{\phi}$ is a homeomorphism by Proposition~\ref{path-homeo-from-conj-prop}, it follows that $H$ is continuous.

Finally, to see that $H$ is a homeomorphism, we simply define $H^{-1} : \G_F \to \G_E$ by $$H^{-1}(\alpha \gamma, l(\alpha) - l(\beta), \beta \gamma) := (\widetilde{\phi}^{-1} (\alpha \gamma), l(\alpha) - l(\beta), \widetilde{\phi}^{-1} (\beta \gamma))$$ and by a symmetric argument, we obtain that $H^{-1} : \G_F \to \G_E$ is a well-defined continuous map.  It is straightforward to check that $H^{-1}$ is an inverse for $H$, and thus $H$ is a homeomorphism.
\end{proof}

\begin{corollary} \label{graph-C*-iso-cor}
Let $E$ and $F$ be countable graphs with no sinks and no sources.  If $X_E \cong X_F$, then $C^*(E) \cong C^*(F)$.
\end{corollary}

\begin{proof}
This follows from the fact that for any graph $E$, the groupoid $C^*$-algebra $C^*(\mathcal{G}_E)$ is isomorphic to the graph $C^*$-algebra $C^*(E)$, and from the fact that isomorphic groupoids produce isomorphic groupoid $C^*$-algebras.
\end{proof}

\subsection{Leavitt Path Algebras of Countable Graphs} \label{LPA-subsec}

Leavitt path algebras are algebraic counterparts of graph $C^*$-algebra and they have attracted a great deal of attention in the past 8 years \cite{AAP}.  Here we show that if two countable graphs have shift spaces that are conjugate, then the Leavitt path algebras of the graphs are isomorphic.

\begin{definition} \label{Leavitt-def}
Let $E$ be a directed graph, and let $K$ be a field.  The \emph{Leavitt path algebra of $E$ with coefficients in $K$}, denoted $L_K(E)$,  is the universal $K$-algebra generated by a set $\{v : v \in E^0 \}$ of pairwise orthogonal idempotents, together with a set $\{e, e^* : e \in E^1\}$ of elements satisfying
\begin{enumerate}
\item $s(e)e = er(e) =e$ for all $e \in E^1$
\item $r(e)e^* = e^* s(e) = e^*$ for all $e \in E^1$
\item $e^*f = \delta_{e,f} \, r(e)$ for all $e, f \in E^1$
\item $v = \displaystyle \sum_{\{e \in E^1 : s(e) = v \}} ee^*$ whenever $0 < |s^{-1}(v)| < \infty$.
\end{enumerate}
\end{definition}

\begin{theorem}
Let $E$ and $F$ be countable graphs with no sinks and no sources.  If $X_E \cong X_F$, then $L_\C(E) \cong L_\C(F)$.
\end{theorem}

\begin{proof}
If $X_E$ and $X_F$ are conjugate, then Theorem~\ref{shifts-conjugate-implies-groupoids-iso-thm} implies that the groupoids $\G_E$ and $\G_F$ are isomorphic.  It follows that the algebras $A(\G_E)$ and $A(\G_F)$ (as defined in \cite[Definition~3.3]{CFST}) are isomorphic.  However, \cite[Proposition~4.3]{CFST} and \cite[Remark~4.4]{CFST} imply that $A(\G_E) \cong L_\C(E)$ and $A(\G_F) \cong L_\C(F)$.
\end{proof}

\subsection{$C^*$-algebras and Leavitt Path Algebras of Row-Finite Graphs} \label{row-finite-graph-alg-subsec}

In this section we prove that row-finite graphs with conjugate edge shifts give rise to isomorphic $C^*$-algebras.  In this setting, we explicitly define the isomorphism (Theorem~\ref{ConjIso}).  Corollary~\ref{graph-C*-iso-cor} in Section~\ref{C*-alg-groupoid-subsec} is similar but distinct.  Theorem~\ref{ConjIso} and Corollary~\ref{graph-C*-iso-cor} both assume that the underlying graphs have no sinks.  Corollary~\ref{graph-C*-iso-cor} applies in the countable graph setting but assumes that the underlying graphs have no sources.  By contrast, Theorem~\ref{ConjIso} applies to row-finite graphs that may contain sources and explicitly constructs the isomorphism.  This result together with its generalization (Theorem~\ref{ConjIsoGen}) therefore recovers the classical fact that if two finite graphs with no sinks have conjugate edge shifts, then their Cuntz-Krieger algebras are isomorphic.

\begin{remark}
If $E$ is a row-finite graph, Property~(2) of the Cuntz-Krieger relations in Definition~\ref{graph-C*-def} is redundant.  If $E$ is a row-finite graph, a collection of elements $\{S_e, P_v: e \in E^1, v \in E^0\}$ in a $C^*$-algebra $B$ is a \emph{Cuntz-Krieger $E$-family} if $\{P_v: v \in E^0 \}$ is a collection of mutually orthogonal projections, $\{S_e: e \in E^1\}$ is a collection of partial isometries, and the following two relations are satisfied:
\begin{enumerate}
  \item[(CK1)] $S^*_e S_e= P_{r(e)}$ for all $e \in E^1$
  \item[(CK2)] $P_v = \Sigma_{s(e)=v} S_e S_e^*$ whenever $v \in E^0$ is not a sink.
\end{enumerate}
\end{remark}

In proof of the following theorem it will be convenient for us to use the following notation:  If $E = (E^0, E^1, r, s)$ is a graph and $v \in E^0$, we write $$vE^n := \{ \alpha \in E^n : s(\alpha) = v \}.$$

\begin{theorem} \label{ConjIso}
Let $E$ and $F$ be row-finite directed graphs with no sinks.  Let $\{s_e, p_v\}$ be a generating Cuntz-Krieger $E$-family for $C^*(E)$ and $\{t_e, q_v\}$ be a generating Cuntz-Krieger $F$-family for $C^*(F)$.  If $\psi: X_F \to X_E$ is a $1$-block conjugacy with a bounded inverse $\phi : X_E \to X_F$, and if $\Phi : B_n(X_E) \to F^1$ is the $n$-block map associated to $\phi$, then there exists a $*$-isomorphism $\pi: C^*(E) \to C^*(F)$ with
\begin{align*}
& \pi(s_e)=  \sum_{\substack{g \in F^1\\ \exists \alpha \in r(e)E^{n-1}\\ \Phi(e \alpha)=g}}{t_g} && \mbox{ and } & \pi(p_v)=  \sum_{\substack{g \in F^1\\ \exists \beta \in vE^n\\ \Phi(\beta)=g}}{t_g t_g^*}.
\end{align*}
\end{theorem}

\begin{proof}
Fix $e \in E^1$ and define 
\begin{align*}
& S_e:=  \sum_{\substack{g \in F^1\\ \exists \alpha \in r(e)E^{n-1}\\ \Phi(e \alpha)=g}}{t_g} 
&P_v:= \sum_{\substack{g \in F^1\\ \exists \beta \in vE^n\\ \Phi(\beta)=g}}{t_g t_g^*}.
\end{align*}
We have
\begin{align*}
  S^*_e S_e &= \sum_{\substack{g \in F^1\\ \exists \alpha \in r(e)E^{n-1}\\ \Phi(e \alpha)=g}}{t_g}^* \sum_{\substack{h \in F^1 \\ \exists \mu \in r(e)E^{n-1} \\ \Phi(e \mu)=h}}{t_h} \\
            &  = \sum_{\substack{g,h \in F^1\\ \exists \alpha,\mu \in r(e)E^{n-1}\\ \Phi(e \alpha)=g \\ \Phi(e \mu)=h}}{t^*_g t_h} 
             = \sum_{\substack{g \in F^1\\ \exists \alpha \in r(e)E^{n-1}\\ \Phi(e \alpha)=g}}{t^*_g t_g} \mbox{ \quad (by the CK-relations)} \\
                    &= \sum_{\substack{g \in F^1\\ \exists \alpha \in r(e)E^{n-1}\\ \Phi(e \alpha)=g}}{q_{r(g)}}
                     = \sum_{\substack{g \in F^1\\ \exists \alpha \in r(e)E^{n-1}\\ \Phi(e \alpha)=g}} \sum_{\substack{h \in F^1\\ s(h)=r(g)}}{t_h t^*_h} \mbox{ \quad (by the CK-relations)}\\
                     &= \sum_{\substack{h \in F^1\\ \exists g \in F^1\\ \exists \alpha \in r(e)E^{n-1}\\s(h)=r(g)\\ \Phi(e \alpha)=g}}{t_h t^*_h}.
\end{align*}
We show that the set 
 \[K: =  \{h \in F^1 : \exists g \in  F^1 \mbox{ and } \alpha \in r(e)E^{n-1} \mbox{ with } s(h)=r(g) \mbox{ and } \Phi(e \alpha)=g \} \]
 is equal to
  \[L:= \{h \in F^1 : \exists \beta \in r(e)E^n \mbox{ with } \Phi(\beta)=h \}.  \]
Let $\Psi$ be the $1$-block map that defines $\psi$.  Since $\psi \circ\phi= \operatorname{Id}_{X_E}$ 
\begin{equation} \label{BMcomp}
\Psi \circ \Phi: B_n(X_E) \to B_1(X_E) \mbox{ is defined by } x_i \dots x_{i+n-1} \mapsto x_i.
\end{equation}  
Let $h \in K$.  Then $ s(h)=r(g)$, which implies $\Psi(g) \Psi(h) \in E^2$.  Observe that $e= \Psi(\Phi (e \alpha))= \Psi(g)$.  There exists $\beta \in E^n$ such that $\Phi(\beta)=h$. Thus $\beta_1 = \Psi(\Phi(\beta))= \Psi(h)$.  Therefore $\Psi(g) \Psi(h)=e \beta_1 \in E^2$.  Hence $h \in L$.  To show the other containment let $h \in L$.  Then $e \beta \in E^{n+1}$.  Define $g:=  \Phi(e \beta_1 \dots \beta_{n-1})$.  So $\Phi(e \beta_1 \dots \beta_{n-1}) \Phi(\beta)=gh \in F^2$, thus $s(h)=r(g)$.  Therefore $h \in K$.  Now we have
\begin{align*}
  S^*_e S_e = \sum_{\substack{h \in F^1\\ \exists g \in F^1\\ \exists \alpha \in r(e)E^{n-1}\\s(h)=r(g)\\ \Phi(e \alpha)=g}}{t_h t^*_h}
                     = \sum_{\substack{h \in F^1\\ \exists \beta \in r(e)E^n\\ \Phi(\beta)=h}}{t_h t^*_h}
                     = P_{r(e)}.
\end{align*}

By definition
\[\sum_{s(e)=v} S_e S_e^* =  \sum_{s(e)=v} \Bigg(\sum_{\substack{g \in F^1\\ \exists \alpha \in r(e)E^{n-1}\\ \Phi(e \alpha)=g}}{t_g}  \sum_{\substack{h \in F^1\\ \exists \mu \in r(e)E^{n-1}\\ \Phi(e \mu)=h}}{t_h^*} \Bigg). \]
Notice that $t_g t_h^*=0$ unless $r(g)=r(h)$ by the Cuntz-Krieger relations. Suppose $g,h \in F^1$ such that $r(g)=r(h)$.  Since $F$ has no sinks there exists a path $x_1 x_2 \ldots \in F^{\infty}$ such that $gx_1 x_2 \ldots, hx_1 x_2 \ldots \in  F^{\infty}$.  The map $\phi$ is surjective so there exists $f_1 f_2 \ldots, b_1 b_2 \ldots \in E^{\infty}$ such that $\phi(ef_1 f_2 \ldots)=gx_1 x_2 \ldots$ and $\phi(eb_1 b_2 \ldots)=hx_1 x_2 \ldots$.  By Equation~\eqref{BMcomp} 
\begin{equation} \label{1stEdge}
f_i= \Psi \circ \Phi (f_i \cdots f_{n+i-1})= \Psi(x_i)=  \Psi \circ \Phi (b_i \cdots b_{n+i-1})= b_i
\end{equation}
for all $i \in \N$.  Therefore $gx_1 x_2 \ldots=\phi(ef_1 f_2 \ldots)=\phi(eb_1 b_2 \ldots)=hx_1 x_2 \ldots$, which implies $g=h$.  Thus $t_g t_h^*=0$ for all $g \neq h$.  Hence
\begin{equation} \label{g=h}
  \sum_{s(e)=v} \Bigg(\sum_{\substack{g \in F^1\\ \exists \alpha \in r(e)E^{n-1}\\ \Phi(e \alpha)=g}}{t_g}  \sum_{\substack{h \in F^1\\ \exists \mu \in r(e)E^{n-1}\\ \Phi(e \mu)=h}}{t_h^*} \Bigg) =  \sum_{s(e)=v} \sum_{\substack{g \in F^1\\ \exists \alpha \in r(e)E^{n-1}\\ \Phi(e \alpha)=g}}{t_g t_g^*}.
\end{equation}
By Equation~\eqref{1stEdge}
\[\sum_{s(e)=v} S_e S_e^* = \sum_{s(e)=v} \sum_{\substack{g \in F^1\\ \exists \alpha \in r(e)E^{n-1}\\ \Phi(e \alpha)=g}}{t_g t_g^*} = \sum_{\substack{g \in F^1\\ \exists e \alpha \in vE^n\\ \Phi(e \alpha)=g}}{t_g t_g^*}= P_v.\]
Therefore $\{S_e, P_v \}$ is a Cuntz-Krieger $E$-family in $C^*(F)$ and there exists a $*$-homomorphism $\pi: C^*(E) \to C^*(F)$ with 
\begin{align*}
& \pi(s_e)=  \sum_{\substack{g \in F^1\\  \exists \alpha \in r(e)E^{n-1}\\ \Phi(e \alpha)=g}}{t_g} && \mbox{ and } & \pi(p_v)=  \sum_{\substack{g \in F^1\\ \exists \beta \in vE^n\\ \Phi(\beta)=g}}{t_g t_g^*}.
\end{align*}

To see that $\pi$ is injective we check that the hypotheses of the gauge-invariant uniqueness theorem hold \cite[Theorem~2.1]{BHRS}.  By assumption, the projections $q_v \in C^*(F)$ are all nonzero.  Observe that for each $v \in E^0$, $\pi(p_v)$ is the sum of mutually orthogonal projections.  So $\pi(p_v)=0$ if and only if each summand is zero, which is not possible.  We know that $C^*(E)$ has a gauge action $\gamma^E$ and $C^*(F)$ has a gauge action $\gamma^F$.  To check that $\pi \circ \gamma_z^E=\gamma^F_z \circ \pi$ it suffices to check on generators.  It is clear that this relation holds for the projections $p_v$.  Note that
\begin{align*}
 \pi \circ \gamma^E_z (s_e) &= \pi(zs_e)= z(\pi(s_e))= z \sum_{\substack{g \in F^1\\ \exists \alpha \in r(e)E^{n-1}\\ \Phi(e \alpha)=g}}{t_g}  = \sum_{\substack{g \in F^1\\ \exists \alpha \in r(e)E^{n-1}\\ \Phi(e \alpha)=g}}{zt_g} \\
 &=  \sum_{\substack{g \in F^1\\ \exists \alpha \in r(e)E^{n-1}\\ \Phi(e \alpha)=g}}{\gamma^F_z (t_g)} = \gamma^F_z \Bigg(  \sum_{\substack{g \in F^1\\  \exists \alpha \in r(e)E^{n-1}\\ \Phi(e \alpha)=g}}{t_g} \Bigg)=  \gamma^F_z \circ \pi (s_e).
\end{align*}
Therefore, the gauge-invariant uniqueness theorem implies that $\pi$ is injective.

To see that $\pi$ is surjective it suffices to show that for all $g \in F^1$ the element $t_g$ is in the image of $\pi$.  Fix $a \in F^1$.  By Equation~\eqref{1stEdge} there exists a unique edge $e \in E^1$ such that for all $\alpha \in E^n$ with $\Phi(\alpha)=a$ we have $\alpha_1=e$.  Observe that
\begin{align*}
\pi \Bigg(s_e & \sum_{\substack{f \in r(e)E^1\\ \exists \alpha \in r(f)E^{n-2}\\ \Phi(ef\alpha)=a}}  s_f s^*_f \Bigg)\\
  &=  \sum_{\substack{b \in F^1\\ \exists \mu \in r(e)E^{n-1}\\ \Phi(e \mu)=b}}{t_b} \sum_{\substack{f \in r(e)E^1\\ \exists \alpha \in r(f)E^{n-2}\\ \Phi(ef\alpha)=a}} \Bigg[  \sum_{\substack{g \in F^1\\ \exists \beta \in r(f)E^{n-1}\\ \Phi(f \beta)=g}}{t_g}  \sum_{\substack{h \in F^1\\ \exists \nu \in r(f)E^{n-1}\\ \Phi(f \nu)=h}}{t_h^*} \Bigg] \\
  &=  \sum_{\substack{b \in F^1\\ \exists \mu \in r(e)E^{n-1}\\ \Phi(e \mu)=b}}{t_b} \sum_{\substack{f \in r(e)E^1\\ \exists \alpha \in r(f)E^{n-2}\\ \Phi(ef\alpha)=a}} \Bigg[ \sum_{\substack{g \in F^1\\ \exists \beta \in r(f)E^{n-1}\\ \Phi(f \beta)=g}}{t_g t_g^*} \Bigg] \mbox{ \quad by Equation~\eqref{g=h}}.
\end{align*}
Notice that $t_a$ is in the sum \[\sum_{\substack{b \in F^1\\ \exists \mu \in r(e)E^{n-1}\\ \Phi(e \mu)=b}}{t_b}.\]  Also observe that in the sum
\[ \sum_{\substack{f \in r(e)E^1\\ \exists \alpha \in r(f)E^{n-2}\\ \Phi(ef\alpha)=a}} \sum_{\substack{g \in F^1\\ \exists \beta \in r(f)E^{n-1}\\ \Phi(f \beta)=g}}{t_g t_g^*} \]
we have that $\alpha= \beta_1 \cdots \beta_{n-2}$ therefore we have
\begin{align*}
 \sum_{\substack{b \in F^1\\ \exists \mu \in r(e)E^{n-1}\\ \Phi(e \mu)=b}}{t_b} \sum_{\substack{f \in r(e)E^1\\ \exists \alpha \in r(f)E^{n-2}\\ \Phi(ef\alpha)=a}} & \Bigg[ \sum_{\substack{g \in F^1\\ \exists \beta \in r(f)E^{n-1}\\ \Phi(f \beta)=g}}{t_g t_g^*} \Bigg] \\
   & = \sum_{\substack{b \in F^1\\ \exists \mu \in r(e)E^{n-1}\\ \Phi(e \mu)=b}}{t_b} \Bigg(  \sum_{\substack{g \in F^1\\ f \in F^1\\ \exists \beta \in r(f)E^{n-1}\\ \Phi(f \beta)=g \\ \Phi(ef \beta_1 \cdots \beta_{n-2})=a}}{t_g t_g^*} \Bigg)\\
&  =  \sum_{\substack{b \in F^1\\ \exists \mu \in r(e)E^{n-1}\\ \Phi(e \mu)=b}}{t_b} \Bigg(  \sum_{\substack{g \in F^1\\ \exists f \beta \in r(e)E^n\\ \Phi(f \beta)=g \\ \Phi(ef \beta_1 \cdots \beta_{n-2})=a}}{t_g t_g^*} \Bigg).
\end{align*}
It is evident that the set 
\[ \{g \in F^1 : \exists f \beta \in r(e)E^n \mbox{ such that } \Phi(f \beta)=g, \mbox{ and }  \Phi(ef \beta_1 \cdots \beta_{n-2})=a \} \] 
is contained in the set $\{ g \in F^1: s(g)=r(a)\}$.  Therefore we show the other containment.  Let $g \in F^1$ such that $s(g)=r(a)$.  Then $ag \in F^2$.  Thus there exists $\nu \in r(e)E^n$ such that $\Phi(e \nu_1 \cdots \nu_{n-1})=a$ and $\Phi(\nu)=g$.  Therefore
\begin{align*}
 \sum_{\substack{b \in F^1\\ \exists \mu \in r(e)E^{n-1}\\ \Phi(e \mu)=b}}{t_b} \Bigg(  \sum_{\substack{g \in F^1\\ \exists f \beta \in r(e)E^n\\ \Phi(f \beta)=g \\ \Phi(ef \beta_1 \cdots \beta_{n-2})=a}}{t_g t_g^*} \Bigg)  &=  \sum_{\substack{b \in F^1\\ \exists \mu \in r(e)E^{n-1}\\ \Phi(e \mu)=b}}{t_b} \Bigg(  \sum_{r(a)=s(g)} t_g t_g^* \Bigg) \\
  &=  \sum_{\substack{b \in F^1\\ \exists \mu \in r(e)E^{n-1}\\ \Phi(e \mu)=b}}{t_b} q_{r(a)} = t_a q_{r(a)} = t_a. 
\end{align*}
\end{proof}

In order to simplify the description of the $C^*$-algebra isomorphism $\pi$ in Theorem~\ref{ConjIso} we assumed that $\psi: X_F \to X_E$ is a $1$-block conjugacy.   This hypothesis can be replaced with the more general assumption that $\psi$ is a bounded conjugacy.   

\begin{definition}
Let $E=(E^0, E^1, r, s)$ be a graph, and let $N \in \N$.  The \emph{$N$\textsuperscript{th} higher block graph} $E^{[N]}$ is the graph whose vertex set is $(E^{[N]})^0 := E^{N-1}$,  whose edge set is $(E^{[N]})^1 := E^{N}$, and for $e_1 \cdots e_{N} \in E^{N}$ the range and source maps are defined by $r^{[N]}(e_1 \cdots e_{N}) := e_2 \ldots e_{N}$ and $s^{[N]}(e_1 \cdots e_{N}) := e_1 \ldots e_{N-1}$.
\end{definition}

\begin{lemma} \label{HBG}
Let $E$ be a row-finite graph with no sinks and let $N \in \N$.  Then $(X_E)^{[N]} = X_{E^{[N]}}$, and $C^*(E) \cong C^*(E^{[N]})$.
\end{lemma}

\begin{proof}
The letters for $(X_{E})^{[N]}$ are the $N$-blocks from $X_E$, which are the paths of length $N$ in $E$.  However, these are also the letters for $X_{E^{[N]}}$.  An infinite sequence of these letters is in either shift space precisely when the letters (i.e., $N$-blocks) overlap progressively.  Thus  $(X_E)^{[N]} = X_{E^{[N]}}$.

The fact that $C^*(E) \cong C^*(E^{[N]})$ is \cite[Theorem~3.1]{Bates}.
\end{proof}

\begin{theorem} \label{ConjIsoGen}
If $E$ and $F$ are row-finite graphs with no sinks, and if $\psi : X_{F} \to X_{E}$ is a bounded conjugacy with bounded inverse, then $C^*(E) \cong C^*(F)$ via an explicit isomorphism.
\end{theorem}

\begin{remark} \label{explicit-iso-rem}
If $\psi$ is a conjugacy and $E$ and $F$ are finite graphs with no sinks, then $\psi$ and its inverse are automatically bounded.
\end{remark}

\begin{proof}
Proposition~\ref{recodeone} and Remark~\ref{rem-recodeone} show that there exists a $1$-block conjugacy $\psi^{[M]}: (X_{F})^{[M]} \to X_{E}$ with bounded inverse. Lemma~\ref{HBG} shows that $(X_{F})^{[M]} = X_{F^{[M]}}$.  Hence by Theorem~\ref{ConjIso} there exists an isomorphism $\pi : C^{*} (E) \to C^{*} (F^{[M]})$.  Lemma~\ref{HBG} also provides an isomorphism $\pi' : C^{*} (F^{[M]}) \to C^{*} (F)$.  Therefore $C^*(E) \cong C^*(F)$ via the explicit isomorphism $\pi' \circ \pi$.
\end{proof}


\begin{remark}
The proofs of Theorem~\ref{ConjIso},  \cite[Theorem~3.1]{Bates}, and Theorem~\ref{ConjIsoGen} go through \emph{mutatis mutandis} for Leavitt path algebras.  Thus similar results hold, and if $E$ and $F$ are finite graphs with no sinks and $X_E$ is conjugate to $X_F$ via a conjugacy $\psi : X_{F} \to X_{E}$, the method described in the proof of Theorem~\ref{ConjIsoGen} can be used to give an explicit isomorphism $\rho : L_K(E) \to L_K(F)$ between the Leavitt path algebras of $E$ and $F$ for any field $K$.
\end{remark}

\begin{corollary}
Let $E$ and $F$ be finite graphs with no sinks, and let $K$ be any field.  If $X_E$ is conjugate to $X_F$, then $L_K(E)$ is isomorphic to $L_K(F)$.
\end{corollary}


\begin{thebibliography}{99}

\bibitem{AAP}
G.~Abrams and G.~Aranda Pino, \emph{The Leavitt path algebra of a graph}, J. Algebra \textbf{293} (2005), 319--334.

\bibitem{Bates} 
T.~Bates, \emph{Applications of the gauge-invariant uniqueness theorem for graph algebras},  Bull. Austral. Math. Soc. \textbf{66} (2002), 57--67.

\bibitem{BHRS}
T.~Bates, J.~H.~Hong, I.~Raeburn, and W.~Szyma\'nski,
\emph{The ideal structure of the $C^*$-algebras of infinite graphs},
Illinois J.~Math \textbf{46}  (2002), 1159--1176.

\bibitem{BPRS} 
T.~Bates, D.~Pask, I.~Raeburn and W.~Szyma\'nski,
\emph{The $C^*$-algebras of row-finite graphs}, New York J.
Math. \textbf{6} (2000), 307--324.

\bibitem{BBG06}
M.~Boyle, J.~Buzzi, and R.~G\'omez, \emph{Almost isomorphism for countable state Markov shifts}, J. Reine Angew. Math. \textbf{592} (2006), 23--47. 

\bibitem{BBG07}
M.~Boyle, J.~Buzzi, and R.~G\'omez, \emph{Good potentials for almost isomorphism of countable state Markov shifts}, Stoch. Dyn. \textbf{7} (2007), 1--15. 

\bibitem{CFST}
L.~O.~Clark, C.~Farthing, A.~Sims, and M.~Tomforde, \emph{A groupoid generalization of Leavitt path algebras}, preprint (2012).

\bibitem{Cun}
J.~Cuntz, \emph{A class of $C^*$-algebras and topological Markov chains II: reducible chains and the Ext-functor for $C^*$-algebras}, Invent.~Math.~\textbf{63} (1981), 25--40.

\bibitem{CK}
J.~Cuntz and W.~Krieger, \emph{A class of $C^*$-algebras and topological Markov chains}, Invent.~Math.~\textbf{56} (1980), 251--268.

\bibitem{CS}
V.~Cyr and O.~Sarig, \emph{Spectral gap and transience for Ruelle operators on countable Markov shifts}, Comm. Math. Phys. \textbf{292} (2009), 637--666. 

\bibitem{EL}
R.~Exel and M.~Laca, \emph{Cuntz-Krieger algebras for
infinite matrices}, J. reine angew. Math. \textbf{512}
(1999), 119--172.

\bibitem{FF}
D.~Fiebig and U.~Fiebig, \emph{Compact factors of countable state Markov shifts}, Theoret. Comput. Sci. \textbf{270} (2002), 935--946. 

\bibitem{FFY}
D.~Fiebig, U.~Fiebig, and M.~Yuri, \emph{Pressure and equilibrium states for countable state Markov shifts},
Israel J. Math. \textbf{131} (2002), 221--257. 

\bibitem{GS}
B.~M.~Gurevich and S.~V.~Savchenko, \emph{Thermodynamic formalism for symbolic Markov chains with a countable number of states}, (Russian) Uspekhi Mat. Nauk \textbf{53} (1998), 3--106; English translation in 
Russian Math. Surveys \textbf{53} (1998), 245--344. 

\bibitem{IY}
G.~Iommi and Y.~Yayama, \emph{Almost-additive thermodynamic formalism for countable Markov shifts}, Nonlinearity \textbf{25} (2012), 165--191. 

\bibitem{LM95} 
D. Lind \& B. Marcus, \emph{An introduction to symbolic dynamics and coding}, Cambridge University Press, 1995.

\bibitem{Massey}
W.~S.~Massey, Algebraic Topology: An Introduction. Reprint of the 1967 edition. Graduate Texts in Mathematics, Vol. 56. Springer-Verlag, New York-Heidelberg, 1977. xxi+261 pp.

\bibitem{MU}
R.~D.~Mauldin and M.~Urba\'nski, \emph{Gibbs states on the symbolic space over an infinite alphabet},
Israel J. Math. \textbf{125} (2001), 93Ð130. 

\bibitem{Pat}
A.~T.~Paterson, \emph{Graph inverse semigroups, groupoids and their $C^*$-algebras}, J. Operator Theory \textbf{48} (2002), 645--662.

\bibitem{PW}
A.~T.~Paterson and A.~E.~Welch, \emph{Tychonoff's theorem for locally compact spaces and an elementary approach to the topology of path spaces}, Proc. Amer. Math. Soc. \textbf{133} (2005), 2761--2770.

\bibitem{PZ}
Y.~Pesin and K.~Zhang, \emph{Phase transitions for uniformly expanding maps}, J. Stat. Phys. \textbf{122} (2006), 1095--1110. 

\bibitem{Pet}
K.~Petersen, \emph{Factor maps between tiling dynamical systems}, Forum Math. \textbf{11} (1999), 503--512.

\bibitem{Sar99}
O.~Sarig, \emph{Thermodynamic formalism for countable Markov shifts}, Ergodic Theory Dynam. Systems \textbf{19} (1999), 1565--1593. 

\bibitem{Sar01}
O.~Sarig, \emph{Phase transitions for countable Markov shifts}, Comm. Math. Phys. \textbf{217} (2001),  555--577. 

\bibitem{Sar06}
O.~Sarig, \emph{Continuous phase transitions for dynamical systems}, Comm. Math. Phys. \textbf{267} (2006), 631Ð667. 

\bibitem{Wag87}
J.~B.~Wagoner, \emph{Markov partitions and $K_2$}, Inst. Hautes ƒtudes Sci. Publ. Math. No. \textbf{65} (1987), 91--129. 

\bibitem{Wag88}
J.~B.~Wagoner, \emph{Topological Markov chains, $C^*$-algebras, and $K_2$}, Adv. in Math. \textbf{71} (1988), 133--185.

\bibitem{Webster-thesis}
S.~Webster, \emph{Directed graphs and $k$-graphs: topology of the path space and how it manifests in the associated $C^*$-algebra}, Ph.D. thesis, University of Wollongong (Australia), 2010.

\bibitem{Web}
S.~Webster, \emph{The path space of a directed graph}, Proc. Amer. Math. Soc., to appear.

\bibitem{Yeend-thesis}
T.~Yeend, \emph{Topological higher-rank graphs, their groupoids and operator algebras}, Ph.D. theses, University of Newcastle (Australia), March 2005.

\bibitem{Yeend}
T.~Yeend, \emph{Groupoid models for the $C^*$-algebras of topological higher-rank graphs}, J.~Operator theory \textbf{57} (2007), 95---120.

\bibitem{You98}
L.~S.~Young, \emph{Statistical properties of dynamical systems with some hyperbolicity}, Ann. of Math. (2) \textbf{147} (1998), 585--650. 

\bibitem{You99}
L.~S.~Young, \emph{Recurrence times and rates of mixing}, Israel J. Math. \textbf{110} (1999), 153--188. 

\end{thebibliography}
\end{document}